\documentclass[11pt,reqno,oneside]{amsart}

\usepackage[utf8]{inputenc}
\usepackage[english]{babel}

\usepackage[a4paper, total={6in, 9in}]{geometry}

\usepackage{pgf,tikz,tikz-cd,pgfplots} \pgfplotsset{compat=newest}
\usetikzlibrary{arrows}
\usetikzlibrary[patterns]
\usepackage{upgreek}
\usepackage{amsmath,amscd, amssymb, amsthm, mathrsfs}
\usepackage[abbrev,nobysame,alphabetic]{amsrefs}
\usepackage{xcolor}
\usepackage{cases}
\usepackage{hyperref}

\usepackage{verbatim}

\newtheorem{Theorem}{Theorem}[section]
\newtheorem*{Theorem*}{Theorem}
\newtheorem{Lemma}[Theorem]{Lemma}

\newtheorem{Proposition}[Theorem]{Proposition}

\newtheorem{Claim}[Theorem]{Claim}

\theoremstyle{definition}

\newtheorem{Remark}{Remark}[section]
\newtheorem*{Remark*}{Remark}

\numberwithin{equation}{section}

\newcommand{\R}{\mathbb{R}} \newcommand{\mathR}{\mathbb{R}}
\newcommand{\Rn}{{\mathR^n}}

\newcommand{\calR}{\mathcal{R}}

\newcommand{\M}{m}  %regularity 

\newcommand{\s}{\hspace{0.5pt}}

\usepackage[pagewise]{lineno} % Line numbers for our draft version
\newcommand*\patchAmsMathEnvironmentForLineno[1]{%
	\expandafter\let\csname old#1\expandafter\endcsname\csname #1\endcsname
	\expandafter\let\csname oldend#1\expandafter\endcsname\csname end#1\endcsname
	\renewenvironment{#1}%
	{\linenomath\csname old#1\endcsname}%
	{\csname oldend#1\endcsname\endlinenomath}}% 
\newcommand*\patchBothAmsMathEnvironmentsForLineno[1]{%
	\patchAmsMathEnvironmentForLineno{#1}%
	\patchAmsMathEnvironmentForLineno{#1*}}%
\AtBeginDocument{%
	\patchBothAmsMathEnvironmentsForLineno{equation}%
	\patchBothAmsMathEnvironmentsForLineno{align}%
	\patchBothAmsMathEnvironmentsForLineno{flalign}%
	\patchBothAmsMathEnvironmentsForLineno{alignat}%
	\patchBothAmsMathEnvironmentsForLineno{gather}%
	\patchBothAmsMathEnvironmentsForLineno{multline}%
	\patchAmsMathEnvironmentForLineno{subequations}%
	\patchAmsMathEnvironmentForLineno{numcases}%
}
\DeclareMathOperator{\divr}{div}
\DeclareMathOperator{\supp}{\mathrm{supp}}
\let\Re\relax
\DeclareMathOperator{\Re}{\mathrm{Re}}

\newcommand{\dif}[1]{\,\mathrm{d}{#1}}

\newcommand{\df}{\mathrm{d}}
\newcommand{\nrm}[2][]{ \| {#2} \|_{#1}}

\allowdisplaybreaks

\newcommand{\abs}[1]{\lvert #1 \rvert}          
\newcommand{\norm}[1]{\lVert #1 \rVert} 
\newcommand{\br}[1]{\langle #1 \rangle} 
\newcommand{\ol}[1]{\overline{#1}}

\newcommand{\p}{\partial}
\newcommand{\eps}{\varepsilon}

\newcommand{\bk}[1]{\left\{ #1 \right\}}

\newcounter{sidenote}
\title[Fixed angle inverse scattering]{Fixed angle inverse scattering for sound speeds close to constant}

\author{Shiqi Ma}
\address{School of Mathematics, Jilin University, Changchun, China}
\email{mashiqi@jlu.edu.cn}

\author{Leyter Potenciano-Machado}
\address{Department of Mathematics and Statistics, University of Jyv\"askyl\"a, Jyv\"askyl\"a, Finland}
\email{leyter.m.potenciano@gmail.com}

\author{Mikko Salo}
\address{Department of Mathematics and Statistics, University of Jyv\"askyl\"a, Jyv\"askyl\"a, Finland}
\email{mikko.j.salo@jyu.fi}

\begin{document}

\begin{abstract}

We study the fixed angle inverse scattering problem of determining a sound speed from scattering measurements corresponding to a single incident wave. The main result shows that a sound speed close to constant can be stably determined by just one measurement. Our method is based on studying the linearized problem, which turns out to be related to the acoustic problem in photoacoustic imaging. We adapt the modified time-reversal method from [P.~Stefanov and G.~Uhlmann, {\it Thermoacoustic tomography with variable sound speed}, Inverse Problems {\bf 25} (2009), 075011] to solve the linearized problem in a stable way, and we use this to give a local uniqueness result for the nonlinear inverse problem.

	\medskip

	\noindent{\bf Keywords:}~~Inverse scattering, wave equation, time-reversal, progressive wave expansion.
	
	{\noindent{\bf 2020 Mathematics Subject Classification:}~~35R30, 35Q60, 35J05, 31B10, 78A40}.
	
\end{abstract}

\maketitle

\section{Introduction and main result} \label{sec:Intro-LS2020}

Inverse scattering problems appear in real-life phenomena and have applications in a wide range of fields such as
radar, sonar, fault detection in fiber optics, geophysical exploration, medical imaging and nondestructive testing. In this work, we study the inverse acoustic scattering problem of recovering a sound speed from fixed angle scattering measurements. Let $n\geq 2$ and $\theta\in {\mathbb{S}}^{n-1}$ be a fixed vector.
Consider $\eta\in C^{1,1}(\mathbb{R}^n)$ and $\eta-1\in H^{s_0}(\R^n)$ with {$s_0>0$ being large enough and} so that $\eta-1$ is compactly supported in $\Omega$ with $\Omega := \{ x\in \R^n \, : \,|x|<1 \}$, i.e.,
\begin{equation} \label{outside:ball}
    \eta(x)=1, \quad |x|\geq 1 - \sigma,
\end{equation}
for a fixed number $\sigma \in (0,1)$.
We also assume that for some $M>1$
\begin{equation}\label{outside:ball_z}
    M^{-1} \leq \eta(x) \leq M, \quad x\in \R^n.
\end{equation}
Now consider a plane wave solution $\tilde{\mathcal{U}}$ to the wave equation 
\begin{equation} \label{eq:1-LS2020}
    (\eta(x) \, \partial_t^2 -\Delta) \tilde{\mathcal{U}}=0  \; \; \text{in}\; \R^{n+1}, \qquad \tilde{\mathcal{U}}|_{\left\{t<-1\right\}} = \delta (t-x\cdot \theta).
\end{equation}
This model describes the propagation of scattered sound waves in an inhomogeneous medium (whose properties are described by the coefficient $\eta$) produced by the interaction of an incident plane wave of the form $\delta (t-x\cdot \theta)$ with the medium. If $c(x)$ is the sound speed in the medium then $\eta(x) = c(x)^{-2}$, but for simplicity we will refer to $\eta$ as the sound speed.

The aim of this paper is to prove that the sound speed $\eta$ is uniquely determined by boundary measurements of a solution to \eqref{eq:1-LS2020} corresponding to a fixed direction $\theta\in {\mathbb{S}}^{n-1}$.
To simplify computations, we assume that $\theta=e_n$, where $e_n$ stands for the $n$th vector of the standard basis in $\R^n$.
For $x \in \R^n$ we write $x = (y, z)$ with $y \in \R^{n-1}$ and $z \in \R$, thus $x\cdot \theta=z$.
From the mathematical viewpoint, for $T>1$, the measurements are encoded by the map
\begin{equation*} %\label{def:B_eta_z_z_z}
	\tilde{\mathcal A}: (\eta-1) \mapsto \tilde{\mathcal{U}}|_{\Sigma_T \cap \{t >z\}},
\end{equation*}
where $\tilde{\mathcal{U}}$ is the solution to \eqref{eq:1-LS2020} and $\Sigma_T$ stands for the lateral boundary of the space-time cylinder $\Omega\times (-T, T)$, that is 
\[
\Sigma_T := \partial \Omega \times (-T, T).
\]
In Proposition \ref{prop:well_posedness_z_1} we prove that there is a unique solution $\tilde{\mathcal{U}}$ to \eqref{eq:1-LS2020} in a suitable Hilbert space so that the restriction on $\Sigma_T \cap \{t >z\}$, denoted by $\tilde{\mathcal{U}}|_{\Sigma_T \cap \{t > z\}}$, is well defined.
We restrict the measurements to $\{ t > z \}$ since the scattered wave vanishes when $t < z$.

In this framework, our first main result states that if $\tilde{\mathcal A}(\eta-1) = \tilde{\mathcal A}(0)$ where $\eta$ satisfies \eqref{outside:ball}--\eqref{outside:ball_z} and is close to $1$, then $\eta \equiv 1$.
Moreover, we also derive the corresponding quantification (stability estimate) with a modulus of continuity of H\"older type.
The space $H^{-k}((-T,T), H^{\alpha}(\mathbb{S}^{n-1}))$ below is defined as the dual space of $H^k_0((-T,T), H^{-\alpha}(\mathbb{S}^{n-1}))$.

\begin{Theorem}\label{th:main_result_z}
Let $n\geq 2$ and $M > 1$. Fix some $s_0\gg n/2+2$. There exist a small $\varrho>0$, $\mu=\mu(n)\in (0,1)$, $T = T(n) > 1$, $s_1 > s_0$, and $C >0$ such that
\[
\norm{\eta-1}_{H^{s_0}(\R^n)}\leq C \norm{\tilde{\mathcal A}(\eta-1)- \tilde{\mathcal A}(0)}_{ H^{-5}((-T, T); H^{1/2}({\mathbb{S}}^{n-1}))|_{\{ t > z \}}}^{\mu},
\]
for all $\eta\in C^{1,1}(\mathbb{R}^n)$ satisfying \eqref{outside:ball}--\eqref{outside:ball_z} and  $\norm{\eta-1}_{H^{s_0}(\R^n)}\leq \varrho$, $\norm{\eta-1}_{H^{s_1}(\R^n)} \leq M$.
\end{Theorem}

Theorem \ref{th:main_result_z} is an instance of a fixed angle inverse scattering result, where one determines a sound speed from scattering measurements corresponding to a single incident plane wave.
There are several results of this type for determining a time-independent potential $q$ instead of a sound speed. The equation in this case is 
\begin{equation} \label{initial_eq_z_z_z_1}
(\partial_t^2 -\Delta + q(x)) \mathcal{W}=0  \; \; \text{in}\; \R^{n+1}, \qquad \mathcal{W}|_{\left\{t \ll 0 \right\}} = \delta (t-x\cdot \theta).
\end{equation}
One can alternatively work on the frequency side with the Schr\"odinger equation
\begin{equation*}
	\left\{\begin{aligned}
	& (-\Delta - k^2 + q(x)) u = 0  \ \text{in} \ \Rn, \\
	& u \text{~is outgoing}.
	\end{aligned}\right.
\end{equation*}
The fixed angle scattering problem consists of determining $q$ from the knowledge of $\mathcal{W}|_{\Sigma_T}$, or equivalently from the scattering amplitude $a_q(k, \theta, \omega)$ for all $k > 0$ and $\omega \in {\mathbb{S}}^{n-1}$. This equivalence is discussed in detail in \cite{RakeshSalo2}.

\smallskip

There are several known results related to recovering small or generic potentials and singularities from fixed angle measurements \cites{BarceloEtAl, BaylissLiMorawetz, Meronno_thesis, Ruiz, Stefanov_generic}.
In the recent works \cites{RakeshSalo2, RakeshSalo1} it was shown that a potential $q \in C^{\infty}_c(\R^n)$ is uniquely determined by measurements corresponding to two incident plane waves from opposite directions $\theta = \pm \theta_0$, or just a single incident plane wave if $q$ satisfies some symmetry conditions.
This result was extended in \cite{MaSa20} to the case when $-\Delta$ is replaced by the Laplace-Beltrami operator $-\Delta_g$ in \eqref{initial_eq_z_z_z_1}, with $g$ being a known metric satisfying certain symmetry conditions.
Using similar ideas, in \cite{MePoSa20} the authors proved analogous results in the case of time-independent first order coefficients. 
We also mention the recent work \cite{KR20}, which studies fixed angle scattering for time-dependent coefficients also in the case of first-order perturbations.
However, the problem of determining a general potential $q \in C^{\infty}_c(\R^n)$ from fixed angle measurements remains open and so does the corresponding inverse backscattering problem (see \cite{RakeshUhlmann} for more information).

\smallskip

The purpose of the present article is to study the fixed angle problem for determining a sound speed $\eta$ instead of a potential $q$. The method in \cites{RakeshSalo2, RakeshSalo1}, which is based on Carleman estimates and reflection arguments, requires symmetry and appears to break down for most nonconstant sound speeds.
In this work we approach the problem for sound speeds by studying the linearized problem.
The main observation is that the linearization of the fixed angle inverse problem for a sound speed has similar features as the acoustic problem in thermo/photoacoustic tomography \cite{SU13multi}.
We then adapt the modified time-reversal method introduced in \cite{SU09} to our case and establish uniqueness, stability and reconstruction for the linearized problem at a constant sound speed.
This can be used to prove local uniqueness and stability for the nonlinear inverse problem as in \cite{SU12}, leading to Theorem \ref{th:main_result_z}. 

\begin{Remark}
It is possible that the methods in this work can be extended to deal with the linearized problem at a general sound speed, and to obtain a counterpart of Theorem \ref{th:main_result_z} showing that $\norm{\eta_1-\eta_2} \leq C \norm{\tilde{\mathcal{A}}(\eta_1-1) - \tilde{\mathcal{A}}(\eta_2-1)}^{\mu}$ in suitable norms when both $\eta_1$ and $\eta_2$ are close to some fixed nonconstant sound speed. These questions are more involved and will be left to a future work.
\end{Remark}

We now describe our method in more detail. As mentioned above, the proof of Theorem \ref{th:main_result_z} reduces to studying the injectivity and stability properties of the linearization of the map $\widetilde{\mathcal{A}}$ at a constant sound speed and to using a general result in \cite{SU12}. Due to technical reasons, we consider a smoother initial value than the one in \eqref{eq:1-LS2020} and study instead the equation 
\begin{equation} \label{eq:2-LS2020}
    (\eta(x) \, \partial_t^2 -\Delta) {\mathcal{U}}=0  \; \; \text{in}\; \R^{n+1}, \qquad {\mathcal{U}}|_{\left\{t<-1\right\}} = H_1 (t-z),
\end{equation}
where $H_1(s)= s$ when $s\geq 0$ and zero otherwise. We now consider the fixed angle inverse scattering problem associated with \eqref{eq:2-LS2020}, where the measurement operator is given by 
\begin{equation*} %\label{def:B_eta_z}
	\mathcal A(\eta-1)= \mathcal U|_{\Sigma_T \cap \{ t > z \}}.
\end{equation*}
Since $\eta$ is time-independent, one has $\tilde{\mathcal{U}}=\partial_t^2 \s {\mathcal{U}}$ and thus the fixed angle inverse scattering problems for $\mathcal{A}$ and $\tilde{\mathcal{A}}$ are equivalent.

The following result describes the precise function spaces involved (its proof is presented in Appendix \ref{appendix:stationary}).
We mention that the space $H^{-k}((-T,T), H^{\alpha}(\R^{n}))$ below is defined in a standard way as the dual space of $H^k_0((-T,T), H^{-\alpha}(\R^{n}))$.
The duality is with respect to the inner product in $L^2((-T,T), L^{2}(\R^{n}))$.
We say that $f\in H^{-k}((-T,T), H^{\alpha}_{loc}(\R^{n}))$ if $\chi\, f \in H^{-k}((-T,T), H^{\alpha}(\R^{n}))$ for all $\chi\in C^\infty_0(\mathbb{R}^n)$.
See Appendix \ref{appendix:stationary} for more details.

\begin{Proposition}\label{prop:well_posedness_z_1}
	Let $n\geq 2$, $T>1$, $\sigma\in (0,1)$ and $M> 1$. Fix some $s_0>n/2+2$.
	Let $\eta\in C^{1,1}(\mathbb{R}^n)$ with $\eta-1 \in H^{s_0}(\mathbb{R}^n)$ be a function satisfying \eqref{outside:ball}-\eqref{outside:ball_z} and
	\[
	\norm{\eta-1}_{H^{s_0}(\R^n)} \leq M, \quad  M^{-1}\leq \eta(x)\leq M  \quad \text{a.e in} \ \    \R^n.
	\]
	There exists a unique distributional solution ${\mathcal{U}}\in H^{-1}((-T, T); H^1_{loc}(\mathbb{R}^n))$ to \eqref{eq:2-LS2020}. Moreover, for any $\chi\in C^\infty_0(\mathbb{R}^n)$ one has
	\begin{equation}\label{estimate_local_z}
	\norm{\chi\, {\mathcal{U}}}_{H^{-1}((-T, T); H^1(\mathbb{R}^n))} \leq C(n, T, M, \norm{\chi}_{C^1}).
	\end{equation}
	Moreover, $\tilde{\mathcal{U}}=\partial_t^2\s\s {\mathcal{U}} \in H^{-3}((-T, T); H^1_{loc}(\mathbb{R}^n))$ is the unique solution to \eqref{eq:1-LS2020}. In particular, for $\chi\in C^\infty_0(\mathbb{R}^n)$ with $\chi(x)=1$ when $|x|< 2$ and $\chi(x)=0$ when $|x|>3$, we have
	\begin{equation}\label{estimate_local_z_z}
	\begin{aligned}
	    	\norm{\tilde{\mathcal{U}}|_{\Sigma_T}}_{H^{-3}((-T, T); H^{1/2}({\mathbb{S}}^{n-1}))} & \lesssim \norm{\chi\,\tilde{\mathcal{U}}}_{H^{-3}((-T, T); H^1(\mathbb{R}^n))} \\
	    	&\lesssim \norm{\chi\,\mathcal{U}}_{H^{-1}((-T, T); H^1(\mathbb{R}^n))}\leq C(n, T, M, \norm{\chi}_{C^1}).
	\end{aligned}
	\end{equation}
\end{Proposition}

Since our approach involves a linearization argument, 
we shall first study the Fr\'echet derivative of $\mathcal{A}$ at the constant $0$, denoted by $A_0$. It is proved in Section \ref{proof_main:result_z} that $A_0$ is given by
\begin{equation} \label{eq:A0-LS2020}
    A_0(f) = U|_{\Sigma_T \cap \{ t > z \}},
\end{equation}
where $U$ solves
 \begin{equation}  \label{eq:s1-LS2020_z}
(\partial_t^2 -\Delta)U = -f(x) \delta (t- z)   \; \; \text{in}\; \R^{n+1}, \qquad U|_{\left\{t<-1\right\}} = 0.
\end{equation}
By employing the progressive wave expansion method, one can further show that any solution to \eqref{eq:s1-LS2020_z} can be written as
\[
U(t,x)= u(t,x)H(t-z),
\]
where $H$ stands for the Heaviside function and $u$ is a $C^2$ function in the set $\{t\geq z \}$ solving the equation 
\begin{equation} \label{eq:utzf-LS2020}
    \left\{ \begin{aligned}
    (\partial_t^2 -\Delta) u & = 0 & \text{in}\; \left\{ t> z\right\}, \\ 
    (\partial_t + \partial_z) u & = -\frac{1}{2} f &\text{on } \left\{ t = z\right\}.
    \end{aligned}\right.
\end{equation}
Thus the linearized inverse problem associated with $A_0$ amounts to determining the initial value $f(x)$ from the knowledge of $u|_{\Sigma_T \cap \{ t > z \}}$.

It is worth mentioning that the linearized problem above is similar to the acoustic problem in thermo/photoacoustic tomography. There one needs to recover an initial condition $g(x)$ from the boundary measurement $v|_{\p \Omega \times (0,T)}$, where $v$ solves 
\begin{equation} \label{bm_1}
\left\{ \begin{aligned}
(\partial_t^2 -\Delta) v & = 0 & \text{in}\; \R^n \times (0,T), \\ 
v & = g & \text{on}\; \left\{ t=0 \right\}, \\
\p_t v &= 0 & \text{on}\; \left\{ t=0 \right\}.
\end{aligned}\right.
\end{equation}
This problem was studied in detail in \cite{SU09}, also for general sound speeds, by using a modification of the time-reversal method.
Error estimates and reconstruction formulae in case of constant sound speeds, and numerical implementations in the case of non-trapping sound speeds, were obtained earlier in \cite{HKN08} and \cite{Hri09} by using the time-reversal method.
For more details on thermo/photoacoustic tomography inverse problems we refer the readers to \cites{KK08,KK11,SU13multi,AKK17}
and the references therein.
We will adopt the modified time-reversal method to show uniqueness, stability and reconstruction for the inverse problem associated with \eqref{eq:utzf-LS2020}. The main difference with previous results is that the initial data is given on the characteristic set $\{ t = z \}$ instead of the standard set $\{ t = 0 \}$. This creates various difficulties, and in order to overcome these we employ energy estimates in space-time domains that are adapted to this characteristic geometry.

\smallskip

We point out that related inverse problems associated with equations similar to \eqref{bm_1} have been studied in different settings. %, like time-cylinder with free and bounded spatial spaces. 
Several authors have studied the identification of unknown sound speeds and first and zero-order potentials associated with the wave equation in different cases, including the knowledge of the solutions restricted to $\{ t = 0 \}$ and to open subsets of $\partial \Omega \times (0,T)$ or $\Omega \times (0,T)$.
In these cases, reconstructive algorithms were proposed in \cites{KM91,Aco15,AM15,BFO20}.
The proofs are based on Carleman estimates, numerical analysis (finite element method with discrete Carleman estimates), time reversal methods for constructing pseudo-inverse operators through Neumann series (as in the present work), and observability inequalities from Control Theory combined with the so-called Geometric Control Condition for wave-like equations.

 \smallskip
 
Returning to our case, the following result is the precise statement of uniqueness and stability in the linearized inverse problem with respect to suitable norms, see \eqref{eq:A0-LS2020}.
We refer to Proposition \ref{prop:RF-LS2020} for a reconstruction formula involving a Neumann series.

\begin{Proposition} \label{prop:fA1f-LS2020}
    Let $s_0\gg n/2+2$, $M>1$ and $T>1$. There exists $C>0$ so that
    \[
    \norm{f}_{L^2(\Omega)}\leq C (\norm{A_0 f}_{H^1(\Sigma_T \cap \{ t > z \})} + \frac {T^{1/2}} {\sigma^{1/4}} \nrm[H^1(\Sigma_T \cap \Gamma)]{A_0 f} ),
    \]
    for $f \in {H}^{s_0}_{\ol{\Omega}}(\mathbb{R}^n)$ with $\supp f \in \{ x \in \Rn \,;\, |x| \leq 1 - \sigma \}$.
\end{Proposition}

We outline the method for proving Proposition \ref{prop:fA1f-LS2020}.
Instead of $A_0$, it is convenient to work with a closely related operator $A_0': F \mapsto u|_{\Sigma}$ where $u$ solves
\begin{equation} \label{eq:vL01-LS2020}
    \left\{ \begin{aligned}
    (\partial_t^2 -\Delta) u & = 0 & \text{in}\; \left\{ t> z\right\}, \\ 
    u & = F &\text{on } \left\{ t = z\right\}.
    \end{aligned}\right.
\end{equation}
Recall $\Gamma = \Sigma_T \cap \{ t = z \}$.  Define a space
\[
\mathcal{H} := \{ F \in H^1(\{t=z\}) \,:\, \mathrm{supp}(F) \subset \ol{\Gamma} \}
\]
equipped with the $H^1$-norm.
In Section \ref{subsec:B-LS2020} we actually use a slightly different definition of $\mathcal H$, see \eqref{eq:zc-LS2020} and \eqref{eq:zc2-LS2020}, but here for illustration purposes we prefer to keep the definition simple.
One would like to think of $A_0'$ as a bounded operator $\mathcal{H} \to H^1(\Sigma)$.
In fact this holds in the standard case where the initial surface is $\{t=0\}$ instead of $\{t=z\}$ since the trace of $u$ is in $H^1(\Sigma)$
%\footnote{{\color{blue}$H^1(\Sigma), H^1(\{t=z\}), H^1(\Gamma)$?}}
\cites{BaoSymes, FinchRakesh}.
We are not aware of such a result for our slanted case, so we will work with smooth functions instead and use norm estimates with uniform bounds in energy spaces.

After some natural derivations, uniqueness and stability for $A_0$ reduce to uniqueness and stability for $A_0'$. Now we use a time-reversal method as in \cite{SU09} and define an approximate inverse $B$ for $A_0'$ as the map $B: h \mapsto v|_{\Gamma}$, where $v$ solves the Dirichlet problem
\begin{equation} \label{eq:vL0-LS2020-intro}
\left\{\begin{aligned}
(\partial_t^2 -\Delta) v & = 0 && \text{in}\; Q \cap \{ t > z \}, \\ 
v & = h && \text{on } \Sigma, \\
%v & = \tilde{u} && \text{on } \Sigma_-, \\
v = \phi_0, \ v_t & = 0 && \text{in}\; Q \cap \{ t=T \}, \\
\Delta \phi_0 & = 0 && \text{~in~} \Omega, \quad \phi_0 = h(\cdot,T) \text{~on~} \partial \Omega,
\end{aligned}\right.
\end{equation}
with $Q := \Omega \times [-1,T]$. We prove that $B$ is a bounded operator $H^1(\Sigma) \to H^1(\Gamma)$.
Note that in odd dimensions, the sharp Huygens' principle implies that any solution $u$ of \eqref{eq:vL01-LS2020} with $F \in \mathcal{H}$ satisfies $u(\,\cdot\,,T)|_{\ol{\Omega}} = 0$ for $T$ large enough. 
Hence, in odd dimensions, letting $h = u|_{\Sigma}$ in \eqref{eq:vL0-LS2020-intro} implies $v = u$ and hence $B A_0' F = F$, so that $B$ is an exact inverse of $A_0'$.
In even dimensions this is no longer true. Instead, we will show that $B$ is almost an inverse of $A_0' $ in the following parametrix sense
\[
B A_0' F = F - \tilde{K} F
\]
where $\tilde{K}$ is a bounded operator on $\mathcal{H}$ with norm strictly less than $1$ when $T$ is large enough.
In \cite{SU09} this was done by using a unique continuation property for the wave equation,
but here in our case due to certain technical issue we use the local energy decay instead. This argument allows us to invert $A_0'$ by a Neumann series and prove Proposition \ref{prop:fA1f-LS2020}.

\smallskip

For readers' convenience we summarize the definitions of the maps $\tilde{\mathcal A}$, $\mathcal A$, $A_0$ and $A_0'$. 
On one side, $\tilde{\mathcal A}$ maps $\eta-1$ to $\tilde{\mathcal{U}}|_{\Sigma_T \cap \{t >z\}}$ with $\delta(t-z)$ as the incident wave, see \eqref{eq:1-LS2020}.
On the other side, $\mathcal A$ maps $\eta-1$ to $\mathcal U|_{\Sigma_T \cap \{t >z\}}$ with $H_1 (t-z)$ as the incident wave, see \eqref{eq:2-LS2020}.
In the same way, $A_0$ is the Fr\'echet derivative of $\mathcal A$ at $0$, see \eqref{eq:A0-LS2020}, and finally $A_0'(-2 \int_{-1}^z f(y,s,s) \dif s) = A_0(f(y,z,z))$ (compare \eqref{eq:utzf-LS2020} with \eqref{eq:vL01-LS2020}).

We also summarize the notation for different sets which we use throughout the paper.
\begin{equation} \label{eq:nost-LS2020}
	\left\{\begin{aligned}
	    \Omega & := \{x\in \R^n \, : \,|x|<1 \}, \quad
	    Q := \Omega \times [-1,T], \quad \Sigma_T := \partial \Omega \times (-T, T), \\
	    \Sigma & := \Sigma_T \cap \{ t > z \}, \quad \Sigma_- := \Sigma_T \cap \{-1 \leq t < z \}, \\
	    \Gamma & = \Sigma_T \cap \{ t = z \}, \quad \Gamma_T := \{(y,z,T) \,;\, (y,z) \in \Omega \}, \\
	    \widetilde Q_\tau & := \{(y,z,t) \,;\, (y,z) \in \Omega,\, z \leq t \leq \tau \} \subset Q, \\
		\widetilde \Sigma_\tau & := \{(y,z,t) \,;\, (y,z) \in \partial \Omega,\, z \leq t \leq \tau \} \subset \Sigma, \\
		\widetilde \Gamma_\tau & := \{(y,z,\tau)\,;\, (y,z) \in \Omega \} \cap \widetilde Q_\tau.
	\end{aligned}\right.
\end{equation}

This paper is structured as follows.
Section \ref{proof_main:result_z} is dedicated to proving Theorem \ref{th:main_result_z} under the assumption that Proposition \ref{prop:fA1f-LS2020} related to linearized problem is known.
In Section \ref{sec:lima-LS2020} we study basic properties of the linear map $A_0$.
In Section \ref{sec:KK-LS2020} we prove Proposition \ref{prop:fA1f-LS2020} as a consequence of a reconstruction formula stated in Proposition \ref{prop:RF-LS2020}.
In the Appendix we give some results on the well-posedness of the forward problem for wave equations in negative Sobolev spaces that are required in our arguments.
The results are stated with finite regularity assumptions on the coefficients, and the dependence of the constants in norm estimates on different quantities is explicitly specified.

\subsection*{Acknowledgements}

S.~M., L.~P-M.~and M.~S.~were supported by the Academy of Finland (Finnish Centre of Excellence in Inverse Modelling and Imaging, grant numbers 312121 and 309963), and M.S.\ was also supported by the European Research Council under Horizon 2020 (ERC CoG 770924).

\section{The nonlinear map. Proof of Theorem \ref{th:main_result_z}} \label{proof_main:result_z}

In this section we will prove Theorem~\ref{th:main_result_z} by using Propositions~\ref{prop:well_posedness_z_1}, \ref{prop:fA1f-LS2020} and~\ref{prop:sol_smooth_1} as well as Lemma~\ref{lemma_u_dttu}. The proofs of these results will be given in later sections. We shall use the following abstract local uniqueness and stability result from~\cite[Theorem 2]{SU12}.

\begin{Proposition} \label{prop:StUh-LS2020}
	Let $\mathcal{X}_j$, $\mathcal{Y}_j$ with $j=1,2,3$ be Banach spaces with $\mathcal{X}_3 \subset \mathcal{X}_1 \subset \mathcal{X}_2$ and $\mathcal{Y}_3 \subset \mathcal{Y}_2 \subset \mathcal{Y}_1$, such that the following interpolation estimates hold:
	\begin{equation}\label{cond:first_z}
	\norm{f}_{\mathcal{X}_1} \lesssim \norm{f}_{\mathcal{X}_2}^{\mu_1}\norm{f}_{\mathcal{X}_3}^{1-\mu_1}, \quad \norm{g}_{\mathcal{Y}_2} \lesssim \norm{g}_{\mathcal{Y}_1}^{\mu_2}\norm{g}_{\mathcal{Y}_3}^{1-\mu_2}, \qquad \mu_1, \mu_2 \in (0,1], \; \; \mu_1 \mu_2 >1/2.
	\end{equation}
	Let $\mathcal{A}: \mathcal{V}_1 \to \mathcal{Y}_1$ be a nonlinear map where $\mathcal{V}_1 \subset \mathcal{X}_1$ is an open subset of $ \mathcal{X}_1$. Consider $f_0\in \mathcal{V}_1$ and assume that 
	\begin{equation}\label{cond:second_z}
	\mathcal{A}(f)= \mathcal{A}(f_0) + A_{f_0}(f-f_0) + R_{f_0}(f), \quad \norm{R_{f_0}(f)}_{\mathcal{Y}_1} \leq C(f_0) \norm{f-f_0}^2_{\mathcal{X}_1}
	\end{equation}
	holds for all $f$ in some neighbourhood of $f_0$ in $\mathcal{V}_1$. Here $A_{f_0}$ stands for the Fr\'echet derivative of $\mathcal{A}$ at $f_0$. In addition, suppose  that 
	\begin{equation}\label{cond:third_z}
	\norm{h}_{\mathcal{X}_2}\leq C \norm{A_{f_0} h}_{\mathcal{Y}_2}, \quad h\in \mathcal{X}_1
	\end{equation}
	and 
	\begin{equation}\label{cond:third_z:z}
	\norm{A_{f_0} h}_{\mathcal{Y}_3}\leq C \norm{h}_{\mathcal{X}_3}, \quad h\in \mathcal{X}_3.
	\end{equation}
	Then for any $L>0$ there exists $\epsilon>0$, so that for any $f$ with
	\begin{equation} \label{cond:f_bounds}
	\norm{f-f_0}_{\mathcal{X}_1} \leq \epsilon, \quad \norm{f}_{\mathcal{X}_3}\leq L,
	\end{equation}
	one has the conditional stability estimate
	\[
	\norm{f-f_0}_{_{\mathcal{X}_1}}\leq CL^{2-\mu_1-\mu_2} \norm{\mathcal{A}(f)-\mathcal{A}(f_0)}^{\mu_1\mu_2}_{\mathcal{Y}_1}.
	\]
	In particular, if $\mathcal{A}(f)=\mathcal{A}(f_0)$ for some $f$ satisfying \eqref{cond:f_bounds}, then $f=f_0$.
\end{Proposition}

The task now is to verify that $\mathcal{A}$ and $A_0$ from Section \ref{sec:Intro-LS2020} satisfy all the conditions of Proposition \ref{prop:StUh-LS2020} in appropriate Banach spaces.
% In fact we will change the setup slightly and write $\eta = 1 + f$, and $\mathcal{A}$ will be considered as the operator $f \mapsto \mathcal{U}|_{\Sigma_T \cap \{ t > z \}}$ (note that $\mathcal{U} = 0$ when $t < z$). Similarly, instead of $A_1$ we consider the operator $A_0:f \mapsto  U|_{\Sigma_T \cap \{ t > z \}}$. In this setting, we will apply Proposition \ref{prop:StUh-LS2020} with  $f_0 \equiv 0$.
We first introduce some useful notations. If $D \subset \R^n$ is closed, define for any $s_0 \in \mathbb{R}$ the set
\[
H^{s_0}_D(\R^n) = \{ f \in H^{s_0}(\R^n) \,:\, \supp(f) \subset D \}.
\]
We also write 
\begin{equation}\label{omega_sigma_z}
\Omega_{\sigma} = \{ x \in \R^n \,:\, \abs{x} < 1-\sigma \}
\end{equation}
where $\sigma \in (0,1)$ is fixed as in \eqref{outside:ball}. Let $s_0>n/2 +2$ and $M>1$. The map $\mathcal{A}$ will be defined in the open subset  
\[
\mathcal{V}_M^{s_0}(\R^n) = \{ f \in H^{s_0}_{\ol{\Omega}_{\sigma}}(\R^n) \,:\, M^{-1} < 1+f < M \}
\]
of $H^{s_0}_{\ol{\Omega}_{\sigma}}(\R^n)$ as the map 
\begin{equation}\label{A:linear-f}
\mathcal{A}: \mathcal{V}_M^{s_0}(\R^n) \to H^{-1}((-T,T); H^{1/2}({\mathbb{S}}^{n-1}))|_{\{ t > z\}}, \ \ \mathcal{A}(f) = {\mathcal{U}}|_{\Sigma_T \cap \{ t > z \}}
\end{equation}
where $f = \eta - 1$.
This is a well-defined map by Proposition \ref{prop:well_posedness_z_1}. We will prove in Lemma \ref{lemma_a_coneone} that the linearization of $\mathcal{A}$ at $0$ is given by 
\begin{equation}\label{A_0:linear}
A_0: H^{s_0}_{\ol{\Omega}_{\sigma}}(\R^n)  \to H^{-1}((-T,T); H^{1/2}({\mathbb{S}}^{n-1}))|_{\{ t > z \}}, \ \ A_0f = U|_{\Sigma_T \cap \{ t > z \}}
\end{equation}
where $U$ solves 
\begin{equation} \label{eq:Ud-LS2020}
\left\{\begin{aligned}
& (\partial_t^2 -\Delta) U = -f(x)\, \delta(t - z)  && \text{in}\; \R^{n+1}, \\
& U|_{\left\{t<-1\right\}} = 0.
\end{aligned}\right.
\end{equation}

Let us verify the conditions of Proposition \ref{prop:StUh-LS2020} one by one. 

\subsection*{Condition \texorpdfstring{\eqref{cond:second_z}}{(2.2)}}

We claim that the map $\mathcal{A}$, defined by \eqref{A:linear-f}, verifies the condition \eqref{cond:second_z} with
\[
\mathcal{X}_1= H^{s_0}_{\ol{\Omega}_{\sigma}}(\R^n), \quad \mathcal{Y}_1=H^{-3}((-T,T); H^{1/2}({\mathbb{S}}^{n-1}))|_{\{ t > z \}}.
\]
Indeed, the proof is contained in the following result, where we write 
\begin{align*}
\hat{\mathcal{A}}: \mathcal{V}_M^{s_0}(\R^n) \to H^{-1}((-T,T); H^{1/2}({\mathbb{S}}^{n-1})), \ \ & \hat{\mathcal{A}}(f) = {\mathcal{U}}|_{\Sigma_T}, \\
\hat{A}_0: H^{s_0}_{\ol{\Omega}_{\sigma}}(\R^n)  \to H^{-1}((-T,T); H^{1/2}({\mathbb{S}}^{n-1})), \ \ &\hat{A}_0f = U|_{\Sigma_T}.
\end{align*}
Thus $\mathcal{A}$ and $A_0$ are the restrictions $\mathcal{A}(f) = \hat{\mathcal{A}}(f)|_{\{ t > z \}}$ and $A_0 f = \hat{A}_0 f|_{\{ t > z \}}$.

\begin{Lemma} \label{lemma_a_coneone}
Let $s_0\gg n/2 +2$, $M>1$ and $T>1$. The map $\hat{\mathcal{A}}$ is well defined
\[
\hat{\mathcal{A}}: \mathcal{V}_M^{s_0}(\R^n) \subset H^{s_0}_{\ol{\Omega}_{\sigma}}(\R^n) \to H^{-1}((-T,T); H^{1/2}({\mathbb{S}}^{n-1})).
\]
Moreover, it is $C^{1,1}$ near $0$ as a map $\mathcal{V}_M^{s_0}(\R^n) \to H^{-3}((-T, T); H^{1/2}({\mathbb{S}}^{n-1}))$, so that 
\begin{equation} \label{remainder_term_z_z_statement}
\norm{\hat{\mathcal{A}}(f)- \hat{\mathcal{A}}(0) - \hat{A}_0 (f)}_{H^{-3}((-T, T); H^{1/2}({\mathbb{S}}^{n-1}))} \lesssim \norm{f}_{H^{s_0}(\R^n)}^2,
\end{equation}
for all $f\in \mathcal{V}_M^{s_0}(\R^n)$ near the origin.
\end{Lemma}
\begin{proof}
The fact that $\hat{\mathcal{A}}$ maps $\mathcal{V}_M^{s_0}(\R^n)$ to $H^{-1}((-T,T); H^{1/2}({\mathbb{S}}^{n-1}))$ is an immediate consequence of Proposition \ref{prop:well_posedness_z_1}. Let us move to prove the $C^{1,1}$ regularity near $0$. Fix $\eta=1+f$ with $f\in \mathcal{V}_M^{s_0}(\R^n)$ near the origin. Let $\hat{A}_0 f = U|_{\Sigma_T}$ where $U$ is as in \eqref{eq:Ud-LS2020}, and let $\calR$ be defined by 
\[
\hat{\mathcal{A}}(f)= \hat{\mathcal{A}}(0) + \hat{A}_0 (f) + \calR|_{\Sigma_T}, \quad \calR|_{\Sigma_T}:=\hat{\mathcal{A}}(f)- \hat{\mathcal{A}}(0) - \hat{A}_0 (f).
\]
%We shall verify that
%\begin{equation} \label{remainder_term_z_z}
%\norm{\calR|_{\Sigma_T}}_{H^{-3}((-T, T); H^{1/2}(S^{n-1}))} \lesssim \norm{f}_{H^{s_0}_{\ol{\Omega}}(\R^n)}^2.
%\end{equation}
In order to prove \eqref{remainder_term_z_z_statement}, we will actually prove a stronger estimate
\begin{equation*} %\label{remainder_term_z_z}
\norm{\calR|_{\Sigma_T}}_{H^{-3}((-T, T); H^{1/2}({\mathbb{S}}^{n-1}))} \lesssim \norm{f}_{L^{\infty}(\R^n)}^2, \quad f\in H^{s_0}_{\ol{\Omega}}(\R^n),
\end{equation*}
which trivially implies the required estimate by using Morrey's inequality since $s_0>n/2$. To do that, we first set
\begin{equation*}%\label{A_0_f_z}
\hat{\mathcal{A}}(f)= \mathcal{U}|_{\Sigma_T}, \quad \hat{\mathcal{A}}(0)= \mathcal{U}_0|_{\Sigma_T}, \quad \hat{A}_0(f)=U|_{\Sigma_T},
\end{equation*}
where the distributions $\mathcal{U}$, $\mathcal{U}_0$ and $U$ satisfy
\begin{equation*}
\begin{aligned}
(\eta (x)\partial_t^2 -\Delta)\s\mathcal{U}& =0, & \; \; \text{in}\; \R^{n+1}, & \quad \mathcal{U}|_{\left\{t<-1\right\}} = H_1 (t-z),\\
(\partial_t^2 -\Delta)\s \mathcal{U}_0& =0, & \; \; \text{in}\; \R^{n+1}, & \quad \mathcal{U}_0|_{\left\{t<-1\right\}} = H_1 (t-z),\\
(\partial_t^2 -\Delta)U&=-f \, \delta(t-z), & \; \; \text{in}\; \R^{n+1},& \quad U|_{\left\{t<-1\right\}} = 0.
\end{aligned}
\end{equation*}
By Proposition \ref{prop:well_posedness_z_1}, we deduce that $\mathcal{U} \in H^{-1}((-T,T); H^1(\R^n))$. By uniqueness of distributional solutions, we have $\mathcal{U}_0(y,z,t)= H_1(t-z)$. Furthermore, Proposition \ref{prop:sol_smooth_1} ensures that $U\in H^{-1}((-T, T); L^2(\R^n))$. On the other hand, a straightforward computation shows that $\mathcal{R}:= \mathcal{U}-\mathcal{U}_0-U$ satisfies in $\mathbb{R}^{n+1}$
\begin{equation} \label{eq:ineq_z}
(\partial_t^2- \Delta)\s\mathcal{R}= -f\, \partial_t^2 ( \mathcal{U} - H_1), \qquad \mathcal{R}|_{\left\{t<-1\right\}} = 0.
\end{equation}
In addition, we also have in $\mathbb{R}^{n+1}$
\begin{equation}\label{eq:u-h_1_z}
(\partial_t^2- \Delta)(\mathcal{U} - H_1)= -f\s\partial_t^2\s  \mathcal{U}, \qquad  \mathcal{U} - H_1|_{\left\{t<-1\right\}} = 0.
\end{equation}
Note that the sources on the right of above equations belong to $H^{-3}((-T, T); L^2(\R^n))$. Since $f = f(x)$ is independent of $t$, for all $\alpha\geq 0$ and any arbitrary $F\in H^{-\alpha}((-T, T); L^2(\R^n))$ one has 
\[
\left\| f\s F \right\|_{H^{-\alpha}((-T, T); L^{2}(\R^n))} \lesssim \norm{f}_{L^\infty(\R^n)}\left\| F \right\|_{H^{-\alpha}((-T, T); L^{2}(\R^n))}.
\]
We apply this inequality with $F=\partial_t^2 ( \mathcal{U} - H_1)\in H^{-3}((-T, T); L^2(\R^n))$, see the source on the right hand side of \eqref{eq:ineq_z}. These facts combined with Lemma \ref{eq:forward_problem_z} give that 
\begin{equation} \label{bound:remainder_term_z_1}
\left\|\mathcal{R}\right\|_{H^{-3}((-T, T); H^{1}(\R^n))} \lesssim \norm{f}_{L^\infty(\R^n)}\left\| \partial_t^2 ( \mathcal{U} - H_1)\right\|_{H^{-3}((-T, T); L^{2}(\R^n))}.
\end{equation}
We bound the norm on the left with the help of \eqref{eq:u-h_1_z} and Lemma \ref{eq:forward_problem_z} as follows
\begin{align*}
\norm{\partial_t^2( \mathcal{U} - H_1)}_{H^{-3}((-T, T); L^{2}(\R^n))}& \leq \norm{\partial_t^2( \mathcal{U} - H_1)}_{H^{-3}((-T, T); H^{1}(\R^n))} \\
&\,\leq \norm{\mathcal{U} - H_1}_{H^{-1}((-T, T); H^{1}(\R^n))}\\
& \, \lesssim \norm{ f\s\s \partial_t^2\s   \mathcal{U}}_{H^{-1}((-T, T); L^{2}(\R^n))} \\
& \, \lesssim  \norm{f}_{L^\infty(\R^n)} \left\| \partial_t^2 \s \mathcal{U} \right\|_{H^{-1}((-T, T); L^{2}(\R^n))}.
\end{align*}
This estimate combined with \eqref{bound:remainder_term_z_1} gives
\begin{equation} \label{eq:ret-LS2020}
	\norm{\mathcal{R}}_{H^{-3}((-T, T); H^{1}(\R^n))}
	\lesssim \norm{f}_{L^\infty(\R^n)}^2 \norm{\partial_t^2 \mathcal{U}}_{H^{-1}((-T, T); L^{2}(\R^n))}.
\end{equation}
Finally, by using the trace theorem, \eqref{eq:ret-LS2020} gives the desired estimate for $\mathcal{R}_{|_{\Sigma_T}}$.
%This finishes the proof.
\end{proof}

\noindent{Recall the definition of the subset $\Omega_\sigma$ given in \eqref{omega_sigma_z}.

\subsection*{Condition \texorpdfstring{\eqref{cond:third_z}}{(2.3)}}

By Proposition \ref{prop:fA1f-LS2020}, whose proof is presented in Section \ref{sec:KK-LS2020}, we consider 
\[
\mathcal{X}_2= L^2_{\ol{\Omega}_{\sigma}}(\R^n), \quad \mathcal{Y}_2=H^{3/2}(\Sigma_T \cap \{ t > z \}).
\]
Now for any $f \in \mathcal{X}_1 = H^{s_0}_{\ol{\Omega}_{\sigma}}(\R^n)$, Proposition \ref{prop:fA1f-LS2020} and the trace theorem imply that 
\[
\norm{f}_{\mathcal{X}_2} \leq C \norm{A_0 f}_{\mathcal{Y}_2}.
\]
Thus condition \eqref{cond:third_z} is satisfied.

\subsection*{Condition \texorpdfstring{\eqref{cond:first_z}}{(2.1)}}

We have to choose a pair of Banach spaces $\mathcal{X}_3 \subset \mathcal{X}_1$ and $\mathcal{Y}_3 \subset \mathcal{Y}_2$ so that for some $\mu_1, \mu_2\in (0,1)$ with $\mu_1\mu_2>1/2$ one has
\begin{equation}\label{inter:estimate_z_z}
\norm{f}_{H^{s_0}(\R^n)} \lesssim \norm{f}_{L^2(\R^n)}^{\mu_1}\norm{f}_{\mathcal{X}_3}^{1-\mu_1}, \quad f\in \mathcal{X}_3,
\end{equation}
and 
\begin{equation} \label{int:spaces_z_1}
\norm{g}_{H^{3/2}(\Sigma_T \cap \{ t > z \})}\lesssim \norm{g}_{H^{-3}((-T,T); H^{1/2}({\mathbb{S}}^{n-1}))|_{\{t > z\}}}^{\mu_2}\norm{g}_{\mathcal{Y}_3}^{1-\mu_2}, \quad g\in \mathcal{Y}_3.
\end{equation}
Fix an arbitrary $\mu_1\in (0,1)$. Consider $s_1>s_0$ satisfying $s_0=0 (\mu_1) + s_1(1-\mu_1)$.
Using complex interpolation, see for instance \cite[Theorem 6.4.5]{BL12}, yields \eqref{inter:estimate_z_z} with
\[
\mathcal{X}_3= H^{s_1}_{\ol{\Omega}_{\sigma}}(\R^n).
\]
On the other hand, if we fix $\mu_2 \in (0,1)$ and choose $s_2 > 3/2$ with $3/2 = (-3) (\mu_2) + s_2 (1-\mu_2)$, then interpolation gives 
\[
\norm{g}_{H^{3/2}(\Sigma_T \cap \{ t > z \})}\lesssim \norm{g}_{H^{-3}(\Sigma_T \cap \{ t > z \})}^{\mu_2} \norm{g}_{H^{s_2}(\Sigma_T \cap \{ t > z \})}^{1-\mu_2}.
\]
To see this, consider $M := \overline{\Sigma_T \cap \{ t > z \}}$ as a compact manifold with smooth boundary and  embed $M$ in some compact manifold $N$ without boundary. If $R$ and $E$ are corresponding restriction and bounded extension operators, then interpolation on $N$ (see e.g.~\cite[Proposition 4.3.1]{Taylor_PDE1}) yields
\begin{align*}
	\norm{g}_{H^{3/2}(M)}
	& = \norm{REg}_{H^{3/2}(M)}
	\leq \norm{Eg}_{H^{3/2}(N)}
	\lesssim \norm{Eg}_{H^{-3}(N)}^{\mu_2} \norm{Eg}_{H^{s_2}(N)}^{1-\mu_2} \\
	& \lesssim \norm{g}_{H^{-3}(M)}^{\mu_2} \norm{g}_{H^{s_2}(M)}^{1-\mu_2}.
\end{align*}
Hence, choosing 
\[
\mathcal{Y}_3= H^{s_2}(\Sigma_T \cap \{ t > z \})
\]
implies \eqref{int:spaces_z_1}. Note that we can make $\mu_1 \mu_2$ as close to $1$ as we want by choosing $s_1$ and $s_2$ large enough.
Therefore, the condition $\mu_1\mu_2>1/2$ is satisfied.

\subsection*{Condition \texorpdfstring{\eqref{cond:third_z:z}}{(2.4)}} We have to prove that
\[
\norm{U}_{H^{s_2}(\Sigma_T \cap \{ t > z \})} \lesssim \norm{f}_{H^{s_1}(\R^n)}.
\]
Indeed, let us take $\alpha=s_2$ and $\beta=s_2+1/2$ in \eqref{cont:A_o_f}. For any fixed $N\in \mathbb{N}$, we consider $s_1, s_2>0$ so that 
\begin{equation}\label{s_+op:m}
2N+ s_2 +5/2<s_1.
\end{equation}
Note that if we further increase $s_1$ in the line after \eqref{int:spaces_z_1}, then $\mu_1$ will increase closer to $1$ and we still have $\mu_1 \mu_2 > 1/2$.
Thus we may assume that \eqref{s_+op:m} holds, and by \eqref{cont:A_o_f} we get
\[
\norm{U}_{H^{s_2}(\Sigma_T \cap \{ t > z \})} \lesssim \norm{f}_{H^{2N+s_2+ 5/2}(\R^n)} \leq \norm{f}_{H^{s_1}(\R^n)}.
\]
Thus condition \eqref{cond:third_z:z} is also satisfied.
%Note that \eqref{s_+op:m} reads as
%\[ 2N+ \frac{3(1+2\mu_2)}{2(1-\mu_2)}+ \frac{5}{2} < \frac{\tre{s_0}}{1-\mu_1} \]
%which always holds by taking $s_0>0$ large enough and fixed $\mu_1$ and $\mu_2$. In particular when $\mu_1\mu_2>1/2$ and $\mu_1\mu_2 \to 1$ as required when verifying \eqref{cond:first_z} in above lines. 

\medskip

We are now in a position to apply Proposition \ref{prop:StUh-LS2020}: there exist $\mu\in (0,1)$, $C = C(L)>0$ and $\varrho>0$ small enough so that 
\[
\norm{f}_{H^{s_0}(\mathbb{R}^n)} \leq C(L) \norm{\mathcal{A}(f) - \mathcal{A}(0)}_{H^{-3}((-T,T); H^{1/2}({\mathbb{S}}^{n-1}))|_{\{ t > z \}}}^{\mu}
\]
whenever $\norm{f}_{H^{s_0}(\mathbb{R}^n)} \leq \varrho$ and $\norm{f}_{H^{s_1}(\mathbb{R}^n)} \leq L$ for $f \in H^{s_1}_{\ol{\Omega}_{\sigma}}(\R^n)$. If we recall that $\mathcal{A}(f) = \hat{\mathcal{A}}(f)|_{\{t > z \}}$ and $\tilde{\mathcal{A}}(1+f) = \p_t^2 \hat{A}(f)$, Lemma \ref{lemma_u_dttu} applied with $k=-5$ implies that 
\[
\norm{\eta-1}_{H^{s_0}(\mathbb{R}^n)} \leq C(L) \norm{\tilde{\mathcal{A}}(\eta-1) - \tilde{\mathcal{A}}(0)}_{H^{-5}((-T,T); H^{1/2}({\mathbb{S}}^{n-1}))|_{\{ t > z \}}}^{\mu}
\]
when $\norm{\eta-1}_{H^{s_0}(\mathbb{R}^n)} \leq \varrho$ and $\norm{\eta}_{H^{s_1}(\mathbb{R}^n)} \leq L$ for $f \in H^{s_1}_{\ol{\Omega}_{\sigma}}(\R^n)$. This finishes the proof of Theorem  \ref{th:main_result_z}.

\section{The linearized map}  \label{sec:lima-LS2020}

We now concentrate on studying the main properties of the linearization at $0$ of the map $\mathcal{A}$ given in \eqref{A:linear-f}. Recall that this linearization is denoted by $A_0$ and it is given by \eqref{A_0:linear} and \eqref{eq:Ud-LS2020}. The existence and uniqueness of solutions to \eqref{eq:Ud-LS2020} is provided by the following result. Its proof is provided in the Appendix. 

\begin{Proposition} \sl \label{prop:sol_smooth_1}
Let $s_0\gg n/2 +2$ and $T>1$. Consider $f \in H^{s_0}_{\ol{\Omega}}(\mathbb{R}^n)$.
There is a unique  distributional solution $U(y,z,t)$ to \eqref{eq:Ud-LS2020}, and it is supported in the region $\{ t\geq z\}$.
In particular, one has
\begin{equation*}
 U(y,z,t) = u(y,z,t)H(t-z),
 \end{equation*}
  where $u$  is a ${C^2}$ function in $\left\{ t\geq z \right\}$ satisfying the IVP \begin{equation} \label{eq:nu_1-LS2020_z_1_z}
\left\{ \begin{aligned}
(\partial_t^2 -\Delta) u & = 0 & \text{in}\; \left\{ t> z\right\}, \\ 
(\partial_t + \partial_z) u & = -\frac{1}{2} f &\text{on } \left\{ t = z\right\},\\
u|_{\left\{t<-1 \right\}}&=0.
\end{aligned}\right.
\end{equation}
In addition, given any $K\geq 3$ we may arrange that $u$ is $C^K$ in the set $\left\{ t\geq z \right\}$ by taking $s_0$ large enough. In particular, one always has $U \in H^{-1}((-T,T), L^2(\R^n))$.
\end{Proposition}
Using standard ODE techniques, one has 
\begin{equation} \label{eq:nu1z-LS2020-uF}
u(y,z,t)_{|_{t=z}}= -\frac{1}{2}\int_{-\infty}^0 f (y, s+z) ds:=F(y,z,z)
\end{equation}
It turns out that instead of \eqref{eq:nu_1-LS2020_z_1_z} it is more convenient to consider the characteristic initial value problem 
\begin{equation} \label{eq:nu1z-LS2020}
\left\{\begin{aligned}
(\partial_t^2 -\Delta) u & = 0 && \text{in~} \{ t > z\}, \\ 
u & = F && \text{on~} \{ t = z\}.
\end{aligned}\right.
\end{equation}
We will apply this with $F$ given in \eqref{eq:nu1z-LS2020-uF}, so that the initial value $F(y,z,z)$ satisfies 
\begin{equation*} %\label{eq:ZFf-LS2020}
    ZF(y,z,z) = -\frac{1}{2} f(y,z).
\end{equation*}
Here $Z$ is defined as $Z F(y,z,z) := \partial_s \big( F(y,s,s) \big) |_{s = z}$. Thanks to the chain rule, one can see that $Z:= \p_t + \p_z$.

\smallskip

As explained in Section \ref{proof_main:result_z}, the Fr\'echet derivative of $\mathcal{A}$ at $0$ is the map $A_0 \colon f \mapsto U|_{\Sigma_T \cap \{ t > z \}}$.
However, it is technically easier to study the map 
\begin{equation} \label{eq:A1-LS2020}
A_0' \colon F \mapsto u|_{\Sigma_T \cap \{ t \geq z \}},
\end{equation}
where $u$ solves \eqref{eq:nu1z-LS2020}. The original map $A_0$ can be recovered by $A_0 f = E_0(A_0' F)$ where $F$ is defined in \eqref{eq:nu1z-LS2020-uF} and $E_0$ denotes extension by zero from $\Sigma_T \cap \{ t \geq z \}$ to $\Sigma_T$.
Therefore in the linearized part we mainly focus on the inverse problem for $A_0'$ instead of $A_0$.
We shall adapt the modified time-reversal method proposed in \cite{SU09} to study $A_0'$.
Recall the notations defined in \eqref{eq:nost-LS2020}.
In the time reversal procedure, we generate an approximate inverse for $A_0'$ as the map $h \mapsto v|_{\Gamma}$, where $v$ solves the problem
\[
(\partial_t^2 -\Delta) v = 0 \ \text{in} \ Q, \quad
v = h \ \text{on} \ \Sigma \cup \Sigma_-, \quad
v = v_0, \ v_t = v_1 \ \text{in} \ \Gamma_T,
\]
with prescribed boundary data $h$ and certain data $(v_0, v_1)$ at the final time $t = T$. For the time reversal argument it would be natural to work with energy spaces. However, it is not obvious that $h \in H^1(\Sigma)$ would imply $v|_{\Gamma} \in H^1(\Gamma)$. We will prove this fact by using energy estimates.
%Then the approximation $\tilde F$ is the restriction of $v$ on the slanted plane $\Gamma$.
%In standard PDEs, one usually consider $v$ horizontally, i.e.~$v \in C([T_1,T_2]; H^1(\Omega))$ (cf \cite[Chapter 7]{evans2010pde}). However, now we are interested in $v|_\Gamma$, and its meaning is not so clear in the literature.For our purpose, we shall define $v|_\Gamma$ in a proper way.
To that end, we first investigate an initial boundary value problem.

\subsection{An initial boundary value problem} \label{subsec:IBVP-LS2020}

For a function $h$, we adopt the following convention:
\begin{equation*} %\label{eq:habv-LS2020}
	h_t := \partial_t h, \quad
	h_z := \partial_z h, \quad
	h_y := \nabla_y h, \quad
	h_x := (\nabla_y h, \partial_z h), \quad
	h_\nu := \nu \cdot h_x,
\end{equation*}
where $\nu \in \mathbb S^{n-1}$ is a unit vector.
We define a seminorm $\nrm[\mathcal H]{\cdot}$ as follows,
\begin{equation} \label{eq:NN0-LS2020}
	\nrm[\mathcal H]{h} = \big( \frac 1 {\sqrt 2} \int_\Gamma (|\nabla_y h|^2 + |h_z + h_t|^2) \dif S \big)^{\frac 1 2}.
	%\Big( \int_{\Omega} |\nabla_x h(y,z,z)|^2 + |h_t|^2 \dif x \Big)^{1/2}.
\end{equation}

We introduce the following PDE,
\begin{equation} \label{eq:vL-LS2020}
	\left\{\begin{aligned}
		(\partial_t^2 -\Delta) v & = G && \text{in}\; Q, \\ 
		v & = u && \text{on } \Sigma \cup \Sigma_-, \\
		v = \phi_0, \ v_t & = \phi_1 && \text{in}\; \Gamma_T.
	\end{aligned}\right.
\end{equation}

We have the following a priori estimate.

\begin{Proposition} \label{prop:Bdd2-LS2020}
	Assume $v \in C^2(\overline{Q})$ solves \eqref{eq:vL-LS2020} with $\phi_0 \in H^1(\Omega)$, $\phi_1 \in L^2(\Omega)$, $G \in L^2(Q)$ and $u|_{\Sigma \cup \Sigma_{-}} \in H^1(\Sigma \cup \Sigma_{-})$.
	For $T \geq 1$, we have
	\begin{equation} \label{eq:vnf-LS2020}
		\nrm[\mathcal H]{v}
		\leq C( \nrm[L^2(\Omega)]{\nabla \phi_0} + \nrm[L^2(\Omega)]{\phi_1} + \nrm[L^2(Q)]{G} + \sqrt{n} e^{T/2} \nrm[H^1(\Sigma)]{v}).
	\end{equation}
	for some constant $C$ independent of $v$, $\phi_0$, $\phi_1$, $G$, $n$ and $T$.
\end{Proposition}

%Note that the prerequisite ``$v \in C^2(\overline{Q})$'' guarantees $v|_\Gamma$ to be meaningful.
To prove Proposition \ref{prop:Bdd2-LS2020}, we do some preparations first.
Let us reproduce the notations $\widetilde Q_\tau$, $\widetilde \Sigma_\tau$ and $\widetilde \Gamma_\tau$ given in \eqref{eq:nost-LS2020}:
\begin{equation*} %\label{eq:TQSG-LS2020}
	\left\{\begin{aligned}
		\widetilde Q_\tau & := \{(y,z,t) \,;\, (y,z) \in \Omega,\, z \leq t \leq \tau \} \subset Q, \\
		\widetilde \Sigma_\tau & := \{(y,z,t) \,;\, (y,z) \in \partial \Omega,\, z \leq t \leq \tau \} \subset \Sigma, \\
		\widetilde \Gamma_\tau & := \{(y,z,\tau)\,;\, (y,z) \in \Omega \} \cap \widetilde Q_\tau,
	\end{aligned}\right.
\end{equation*}
thus $\partial \widetilde Q_\tau = \widetilde \Gamma_\tau \cup \widetilde \Sigma_\tau \cup \Gamma$ for $\tau \geq 1$. Note also that when $-1 \leq \tau < 1$ the set $\widetilde \Gamma_\tau$ is a strict subset of $\Omega \times \{ \tau \}$.
To abbreviate, we denote $\square:= \partial^2_t- \Delta$.

\begin{Lemma} \label{lem:Bdd-LS2020}
	Under the same assumptions as in Proposition \ref{prop:Bdd2-LS2020}, we have
	\[
	\nrm[\mathcal H]{v}^2
	\leq C(\nrm[L^2(\Omega)]{\nabla \phi_0}^2 + \nrm[L^2(\Omega)]{\phi_1}^2 + e^T \nrm[L^2(Q)]{G}^2 + \nrm[L^2(\Sigma)]{v_t} \nrm[L^2(\Sigma)]{v_\nu})
	\]
	for some constant $C$ independent of $u$, $\phi_0$, $\phi_1$, $G$ and $T$.
\end{Lemma}

\begin{proof} %[Proof of Lemma \ref{lem:Bdd-LS2020}]
	%	The existence and uniqueness of the solution $v$ is is guaranteed by Lemma \ref{lem:vEx-LS2020}. 
	
	Integrating the identity
	\begin{equation} \label{eq:rts-LS2020}
	2 \Re \{\overline{v}_t \square v \}
	= \Re \divr_{x,t} \big( -2 \overline{v}_t \nabla_x v, |v_t|^2 + |\nabla_x v|^2 \big)
	\end{equation}
	over $\widetilde Q_\tau$ when $\tau \geq 1$ and noticing that $\square v = G$ in $\widetilde Q_\tau$, we obtain that
	\begin{equation}\label{eq:0g0-LS2020}
	\begin{aligned}
	\Re \int_{\widetilde Q_\tau} 2 \overline{v}_t G
	&= \int_{\widetilde \Gamma_\tau} (|v_t|^2 + |\nabla_x v|^2) \dif S- 2 \Re \int_{\widetilde \Sigma_\tau}  \overline{v}_t v_{\nu} \dif S\\
	& \qquad  - \frac 1 {\sqrt 2} \int_\Gamma (|\nabla_y v|^2 + |v_z + v_t|^2) \dif S,% \label{eq:0g0-LS2020}
	\end{aligned}
	\end{equation}
	Hence, \eqref{eq:0g0-LS2020} becomes
	\begin{equation} \label{eq:0g02-LS2020}
		\Re \int_{\widetilde Q_\tau} 2 \overline{v}_t G
		= \int_{\widetilde \Gamma_\tau} (|v_t|^2 + |\nabla_x v|^2) \dif S - \frac 1 {\sqrt 2} \nrm[\mathcal H]{v}^2 - 2 \Re \int_{\widetilde \Sigma_\tau}  \overline{v}_t v_{\nu} \dif S.
	\end{equation}

	We set
	\begin{equation} \label{eq:0gh-LS2020}
		e(\tau) := \int_{\widetilde \Gamma_\tau} (|\nabla_x v|^2 + |v_t|^2) \dif S.
	\end{equation}
	From \eqref{eq:0g02-LS2020} we deduce
	\begin{align}
		e(\tau)
%		& = \Re \int_{\widetilde Q_\tau} 2 \overline{v}_t G + \frac 1 {\sqrt 2} \nrm[\mathcal H]{v}^2 + 2 \Re \int_{\widetilde \Sigma_\tau} \overline{v}_t v_{\nu} \dif S \nonumber \\
		& \leq \int_0^\tau e(s) \dif s + \int_{\widetilde Q_\tau} |G|^2 + \frac 1 {\sqrt 2} \nrm[\mathcal H]{v}^2 + 2 \int_{\widetilde \Sigma_\tau}  |v_t v_{\nu}| \dif S, \nonumber
	\end{align}
	so the Gronwall's inequality gives
	\begin{equation} \label{eq:0g12-LS2020}
		e(\tau)
		\leq e^\tau (\frac 1 {\sqrt 2} \nrm[\mathcal H]{v}^2 + \int_{\widetilde Q_\tau} |G|^2 + 2 \int_{\widetilde \Sigma_\tau}  |v_t v_{\nu}| \dif S).
	\end{equation}
	
	From \eqref{eq:0g02-LS2020} we also obtain
	\begin{align*}
		& \ \frac 1 {\sqrt 2} \nrm[\mathcal H]{v}^2
		= \int_{\widetilde \Gamma_\tau} (v_t^2 + |\nabla_x v|^2) \dif S - 2 \Re \int_{\widetilde \Sigma_\tau}  \overline{v}_t v_{\nu} \dif S - \Re \int_{\widetilde Q_\tau} 2 \overline{v}_t G \nonumber \\
%		\leq & \ e(\tau) + 2 \int_{\widetilde \Sigma_\tau}  |v_t v_{\nu}| \dif S + \frac 1 \epsilon \int_{\widetilde Q_\tau} |G|^2 + \epsilon \int_0^\tau e(s) \dif s \nonumber \\
%		\leq & \ e(\tau) + 2 \int_{\widetilde \Sigma_\tau}  |v_t v_{\nu}| \dif S + \frac 1 \epsilon \int_{\widetilde Q_\tau} |G|^2 \nonumber \\
%		& \ + \epsilon \int_0^\tau e^s (\frac 1 {\sqrt 2} \nrm[\mathcal H]{v}^2 + \int_{\widetilde Q_s} |G|^2 + 2 \int_{\widetilde \Sigma_s} |v_t v_{\nu}| \dif S) \dif s \qquad (\text{by~} \eqref{eq:0g12-LS2020}) \nonumber \\
%		\leq & \ e(\tau) + 2 \int_{\widetilde \Sigma_\tau}  |v_t v_{\nu}| \dif S + \frac 1 \epsilon \int_{\widetilde Q_\tau} |G|^2 + \epsilon (e^\tau - 1) [\frac 1 {\sqrt 2} \nrm[\mathcal H]{v}^2 + \int_{\widetilde Q_\tau} |G|^2 + 2\int_{\widetilde \Sigma_\tau} |v_t v_{\nu}| \dif S] \nonumber \\
		\leq & \ e(\tau) + 2[1 + \epsilon (e^\tau - 1)] \int_{\widetilde \Sigma_\tau}  |v_t v_{\nu}| \dif S + [\frac 1 \epsilon + \epsilon (e^\tau - 1)] \int_{\widetilde Q_\tau} |G|^2 + \epsilon (e^\tau - 1) \frac 1 {\sqrt 2} \nrm[\mathcal H]{v}^2.
	\end{align*}
	By setting $\epsilon = [2(e^\tau - 1)]^{-1}>0$ and absorbing $\nrm[\mathcal H]{v}^2$ on the right hand side, we obtain
	\begin{equation} \label{eq:0g03-LS2020}
		\frac 1 {2 \sqrt 2} \nrm[\mathcal H]{v}^2
		\leq e(\tau) + 3 \int_{\widetilde \Sigma_\tau} |v_t v_{\nu}| \dif S + 4e^\tau \int_{\widetilde Q_\tau} |G|^2.
	\end{equation}
	Now setting $\tau = T$ and noting that $v(\cdot,T) = \phi_0$ and $v_t(\cdot,T) = \phi_1$, \eqref{eq:0g03-LS2020} gives
	\begin{equation*}
		\nrm[\mathcal H]{v}^2
		\leq C(\nrm[L^2(\Omega)]{\nabla \phi_0}^2 + \nrm[L^2(\Omega)]{\phi_1}^2 + \nrm[L^2(\Sigma)]{v_t} \nrm[L^2(\Sigma)]{v_\nu} + e^T \nrm[L^2(Q)]{G}^2),
	\end{equation*}
	for some constant $C$ independent of $T$.
	We arrive at the conclusion.
\end{proof}

The norm $\nrm[L^2(\Sigma)]{v_\nu}$ can also be estimated.
%It remains to estimate the norm of the normal derivative $v_\nu$.
Following \cite[Lemma 3.3]{RakeshSalo1}, we examine the quantity $(x \cdot \nabla v) \square v$ and use integration by parts.

\begin{Lemma} \label{lem:tdlv-LS2020}
	Assume $v \in C^2(\ol{Q})$ solves \eqref{eq:vL-LS2020}
	For $\tau \geq 1$, we have
	\begin{equation*}
		\nrm[L^2(\widetilde \Sigma_\tau)]{v_\nu}^2
		\leq C ne^\tau (\nrm[L^2(\widetilde Q_\tau)]{G}^2 + \nrm[\mathcal H]{v}^2 + ne^{\tau} \nrm[H^1(\widetilde \Sigma_\tau)]{v}^2),
	\end{equation*}
	for some constant $C$ independent of $v$, $G$, $\tau$ and the dimension $n$.
\end{Lemma}

\begin{proof}
%	As in the proof of Lemma \ref{lem:Bdd-LS2020}, set $\square:= \partial_t^2- \Delta$.
	Integrating the identity
	{\small \[
		2 \Re (x \cdot \nabla \overline v) \square v = \Re \divr_{x,t} \big( x(|\nabla v|^2 - |v_t|^2) - 2(x\cdot \nabla \overline v) \nabla v, 2 (x\cdot \nabla \overline v) v_t \big) + n |v_t|^2 - (n-2)|\nabla v|^2
		\]}
	over $\widetilde Q_\tau$ and noticing that $\square v = G$ in $\widetilde Q_\tau$, similar to \eqref{eq:0g0-LS2020} we now have
	\begin{align}
		2 \int_{\widetilde \Sigma_\tau} |v_\nu|^2
		& = \frac 1 {\sqrt 2} \int_\Gamma [z|\nabla \big( v \big)|^2 - 2\Re \{x\cdot \nabla \big( \overline v \big) \partial_z \big( v \big) \} ] + \Re \int_{\widetilde \Gamma_\tau} 2(x\cdot \nabla v) v_t \nonumber \\
		& \quad - 2\Re \int_{\widetilde Q_\tau} (x \cdot \nabla \overline v) G + \int_{\widetilde \Sigma_\tau} (|v_\nu|^2 + \frac 1 2 \sum_{i \neq j}|\Omega_{ij} v|^2 - |v_t|^2) \nonumber \\
		& \quad + \int_{\widetilde Q_\tau} (n|v_t|^2 - (n-2)|\nabla v|^2), \label{eq:ig0-LS2020}
	\end{align}
	where we used (see \cite[(1.19)]{RakeshSalo1}) the fact
	\[
	|\nabla \varphi(x)|^2 = |\varphi_\nu(x)|^2 + \frac 1 2 \sum_{i \neq j}|\Omega_{ij} \varphi(x)|^2, \quad \text{where} \quad |x| = 1, \ \Omega_{ij} = x_i \partial_j - x_j \partial_i.
	\]
	By moving the $v_\nu$-term on the RHS of \eqref{eq:ig0-LS2020} to the left, we further obtain
	\begin{align}
		\int_{\widetilde \Sigma_\tau} |v_\nu|^2
		& \lesssim
%		\int_\Gamma |\nabla \big( v \big)|^2 + \int_{\widetilde \Gamma_\tau} (|\nabla v|^2 + |v_t|^2) + \nrm[H^1(\widetilde \Sigma_\tau)]{v}^2 + n \int_{\widetilde Q_\tau} (|\nabla v|^2 + |v_t|^2 + |G|^2) \nonumber \\ & = 
		\nrm[\mathcal H]{v}^2 + e(\tau) + \nrm[H^1(\widetilde \Sigma_\tau)]{v}^2 + n \int_0^\tau e(s) \dif s + n \int_{\widetilde Q_\tau} |G|^2, \label{eq:ig1-LS2020}
	\end{align}
	where $e(\tau)$ is defined in \eqref{eq:0gh-LS2020}.
	Combining \eqref{eq:ig1-LS2020} and \eqref{eq:0g12-LS2020}, we have
	\begin{align*}
		\int_{\widetilde \Sigma_\tau} |v_\nu|^2
%		& \lesssim (n+1) e^\tau (\nrm[\mathcal H]{v}^2 + \int_{\widetilde Q_\tau} |G|^2 + 2 \int_{\widetilde \Sigma_\tau}  |v_t v_\nu| \dif S) + \nrm[\mathcal H]{v}^2 + \nrm[H^1(\widetilde \Sigma_\tau)]{v}^2 \nonumber \\
		& \lesssim (n+1) e^\tau (\nrm[\mathcal H]{v}^2 + \int_{\widetilde Q_\tau} |G|^2 + \frac 1 \epsilon \nrm[\widetilde \Sigma_\tau]{v_t}^2 + \epsilon \nrm[\widetilde \Sigma_\tau]{v_\nu}^2) + \nrm[\mathcal H]{v}^2 + \nrm[H^1(\widetilde \Sigma_\tau)]{v}^2.
	\end{align*}
	Setting $\epsilon = [2C(n+1)e^\tau]^{-1}$ for some constant $C$ big enough, we have
	\begin{equation*}
		\int_{\widetilde \Sigma_\tau} |v_\nu|^2
		\lesssim (n+1) e^\tau \int_{\widetilde Q_\tau} |G|^2 + ne^\tau \nrm[\mathcal H]{v}^2 + n^2 e^{2\tau} \nrm[H^1(\widetilde \Sigma_\tau)]{v}^2.
	\end{equation*}
	We arrive at the conclusion.
\end{proof}

Now we are ready to prove Proposition \ref{prop:Bdd2-LS2020}.

\begin{proof}[Proof of Proposition \ref{prop:Bdd2-LS2020}]
	Combining Lemma \ref{lem:Bdd-LS2020} and Lemma \ref{lem:tdlv-LS2020}, we get
	\begin{align*}
		\nrm[\mathcal H]{v}^2
%		& \leq C(\nrm[L^2(\Omega)]{\nabla \phi_0}^2 + \nrm[L^2(\Omega)]{\phi_1}^2 + \frac 1 {4\epsilon} \nrm[L^2(\Sigma)]{v_t}^2 + \epsilon \nrm[L^2(\Sigma)]{v_\nu}^2) \nonumber \\
%		& \leq C(\nrm[L^2(\Omega)]{\nabla \phi_0}^2 + \nrm[L^2(\Omega)]{\phi_1}^2) + \frac C {4\epsilon} \nrm[L^2(\Sigma)]{v_t}^2 \\
%		& \quad + \epsilon C ne^{T} (\nrm[L^2(Q)]{G}^2 + \nrm[\mathcal H]{v}^2 + ne^{T} \nrm[H^1(\Sigma)]{v}^2) \nonumber \\
		& \leq C(\nrm[L^2(\Omega)]{\nabla \phi_0}^2 + \nrm[L^2(\Omega)]{\phi_1}^2) + \epsilon C ne^T \nrm[L^2(Q)]{G}^2 + (\frac C {4\epsilon} + \epsilon C n^2e^{2T}) \nrm[H^1(\Sigma)]{v}^2 \nonumber \\
		& \quad + \epsilon C n e^{T} \nrm[\mathcal H]{v}^2.
	\end{align*}
	By setting $\epsilon = (2C ne^{T})^{-1}$ and absorbing the $\nrm[\mathcal H]{v}^2$ term on the RHS by the LHS, we finally arrive at \eqref{eq:vnf-LS2020}.
	The proof is complete.
\end{proof}

\subsection{The time reversal method} %\label{subsec:TiRe-LS2020}

%For technical reasons we shall replace $T$ by $T$ and work on $\Omega \times (0,T)$ instead of $\Omega \times (0,T)$ in this section. 
Recall the notation from \eqref{eq:nost-LS2020}. Given the data $\{u |_{\Sigma}\}$, we aim to construct an approximation of $F$ in \eqref{eq:nu1z-LS2020} by using a modified time-reversal method as in \cite{SU09}.
%This method takes advantage of the fact that if $u(\cdot,t)$ is solution of the wave equation then $u(\cdot,T-t)$ is too.
We define a function $v$ as the solution of the following system:
\begin{equation} \label{eq:vL0-LS2020}
\left\{\begin{aligned}
(\partial_t^2 -\Delta) v & = 0 && \text{in}\; Q, \\ 
v & = \tilde u && \text{on } \Sigma \cup \Sigma_-, \\
%v & = \tilde{u} && \text{on } \Sigma_-, \\
v = \phi_0, \ v_t & = 0 && \text{in}\; \Gamma_T, \\
\Delta \phi_0 & = 0 && \text{~in~} \Omega, \quad \phi_0 = u(\cdot,T) \text{~on~} \partial \Omega.
\end{aligned}\right.
\end{equation}
Here $\tilde{u}$ is a suitable extension of $u$ from $\Sigma$ to $\Sigma_-$ given in Lemma \ref{lem:vEx-LS2020} below. The choice of this extension will have no influence on the analysis in the rest of the paper.
However, as mentioned earlier, it is not obvious that $v|_\Gamma$ is in $H^1(\Gamma)$ if $u \in H^1(\Sigma)$.
Hence, we shall define
\begin{equation*}
	v|_\Gamma := \lim_{\epsilon \to 0^+} v_\epsilon |_\Gamma
\end{equation*}
where $v_\epsilon$ solves
\begin{equation} \label{eq:vLe-LS2020}
	\left\{\begin{aligned}
		(\partial_t^2 -\Delta) v_\epsilon & = 0 && \text{in}\; Q, \\ 
		v_\epsilon & = u_\epsilon && \text{on } \Sigma \cup \Sigma_-, \\
		v_\epsilon = \phi_0, \ \partial_t v_\epsilon & = 0 && \text{in}\; \Gamma_T, \\
		\Delta \phi_0 & = 0 && \text{~in~} \Omega, \quad \phi_0 = u_\epsilon(\cdot,T) \text{~on~} \partial \Omega.
	\end{aligned}\right.
\end{equation}
Here $u_\epsilon$ is a modification  of $u$ on the lateral boundary given in \eqref{eq:uep-LS2020} such that the compatibility requirements on $\partial \Gamma_T$ required by the existence of a smooth solution $v_{\eps}$ (when $u|_{\Sigma}$ is smooth) are satisfied.
Then we know that $v_\epsilon \in C^2(\overline Q)$ by \cite[Remark 2.10]{llt1986non},
and hence Proposition \ref{prop:Bdd2-LS2020} can be applied to $v_\epsilon$ and the limit exists due to the  estimate given in Proposition \ref{prop:Bdd2-LS2020}.
For convenience we reproduce \cite[Remark 2.10]{llt1986non} in the next lemma.
%\sq{[The paper \cite{llt1986non} only covered regularity, but said nothing about the existence of the solution.]}

\begin{Lemma} \label{lem:LLT210-LS2020}
	Let $\Phi$ be a solution of the system
	\begin{equation*}
		\left\{\begin{aligned}
		(\partial_t^2 -\Delta) \Phi & = F && \text{in}\; \Omega \times [0,T], \\ 
		\Phi & = g && \text{on } \Xi := \partial \Omega \times [0,T], \\
		\Phi = \Phi_0, \ \Phi_t & = \Phi_1 && \text{in}\; \Omega \times \{t=0\},
		\end{aligned}\right.
	\end{equation*}
	with $(F,g, \Phi_0, \Phi_1)$ satisfying the regularity assumptions ($m$ is a non-negative integer)
	\begin{equation*}
		\left\{\begin{aligned}
		& F \in L^1(0,T; H^m(\Omega)), \ \frac {\df^m F} {\df t^m} \in L^1(0,T; L^2(\Omega)), \\
		& \Phi_0 \in H^{m+1}(\Omega), \ \Phi_1 \in H^{m}(\Omega), \\
		& g \in H^{m+1}(\Xi) := L^2(0,T; H^{m+1}(\Xi)) \cap H^{m+1}(0,T; L^2(\Xi))
		\end{aligned}\right.
	\end{equation*}
	and satisfying all necessary compatibility conditions up to order $m$.
	Then
	\begin{equation*}
	\Phi \in C([0,T]; H^{m+1}(\Omega)), \ 
	\frac {\df^{(m+1)} \Phi} {\df t^{(m+1)}} \in C([0,T]; L^2(\Omega)), \ \text{and} \
	\frac {\partial \Phi} {\partial \nu} \in H^m(\Xi).
	\end{equation*}
\end{Lemma}

%Note that $u$ is undefined on $\Sigma_-$ at the moment, and $u$ shall be extended to $\Sigma_-$ in Lemma \ref{lem:vEx-LS2020} bellow.
Note that in contrast with $u$ which is defined in the infinite half plane $\{(y,z,t) \,;\, t \geq z \}$, $v$ is only defined in the finite cylinder $\Omega \times [-1,T]$.

\begin{Lemma} \label{lem:vEx-LS2020}
	If $u|_{\Sigma} \in C^K(\Sigma)$ where $K = \lceil n/2 \rceil + 2$, then there exists a unique solution $v \in C([-1,T]; H^1(\Omega))$ of the system \eqref{eq:vL0-LS2020}, and $v|_{\Gamma} \in H^1(\Gamma)$.
\end{Lemma}

\begin{proof}
	We extend $u$ from $\Sigma$ to $\Sigma_-$ as follows.
	Inspired by the Taylor's expansion, 
%	From Proposition \ref{prop:sol_smooth_1}, we know that $u \in C^K(\{ t \geq z \})$ where $K = \lceil n/2 \rceil + 2$ by taking $s_0$ large enough. 
	we define the following extension of $u$,
	\begin{equation*}
	\tilde u(y,z,t) :=
	\left\{\begin{aligned}
	& u(y,z,t), && \text{~on~} \Sigma, \\
	& \sum_{k = 0}^K (\partial_t^k u)(y,z,z) \cdot (t-z)^k / k!, && \text{~on~} \Sigma_-.
	\end{aligned}\right.
	\end{equation*}
	Then $\tilde u \in C^K(\Sigma \cup \Sigma_-) \subset H^K(\Sigma \cup \Sigma_-)$. By the trace theorem $\tilde u|_{\partial \Omega \times \{T\}}$ is in $H^{K-\frac 1 2}(\partial \Omega \times \{T\})$, which implies that the harmonic function $\phi_0$ in \eqref{eq:vL0-LS2020} is in $H^K(\Gamma_T)$.

	Now we construct a series of approximate Dirichlet boundary data $\{u_\epsilon\}_{\epsilon > 0}$ such that the compatibility conditions needed in Lemma \ref{lem:LLT210-LS2020} will be satisfied.
	To that end, we fix a cutoff function $\chi_{0} \in C_c^\infty(\R)$ satisfying $\chi_0(t) = 1$ when $|t| \leq 1$ and $\chi_0(t) = 0$ when $|t| \geq 2$,
	and we set
	\begin{equation} \label{eq:uep-LS2020}
		u_\epsilon(x,t)
		:= \chi_0(\frac {T-t} \epsilon) \tilde u(x,T) + (1 - \chi_0(\frac {T-t} \epsilon))\tilde u(x,t).
	\end{equation}
	We see $u_\epsilon = \tilde u = u$ on $\Gamma_T$.
	We denote the initial velocity of \eqref{eq:vLe-LS2020} as $\phi_1$, i.e.~$u_t = \phi_1 = 0$ in $\Gamma_T$.
	The compatibility conditions for Lemma \ref{lem:LLT210-LS2020} up to the order $K$ are the following (cf.~e.g.~\cite[\S 7.2 (62)]{evans2010pde}):
	\begin{equation} \label{eq:DPue-LS2020}
	(\Delta^k \phi_0, \Delta^k \phi_1) = (\partial_t^{2k} u_\epsilon, \partial_t^{2k+1} u_\epsilon) \text{~on~} \Gamma_T, \quad \forall k:  0 \leq k \leq \Big \lfloor \frac K 2 \Big \rfloor.
	\end{equation}
	It can be checked that \eqref{eq:DPue-LS2020} is true, and this is simply because $\phi_0 = u_\epsilon(\cdot,T)$, $\Delta \phi_0 = 0$, $\phi_1 = 0$ and $\partial_t^j u_\epsilon(\cdot, T) = 0$ for $\forall j \geq 1$.
	
	It is straightforward to check that
	\begin{equation} \label{eq:ue1-LS2020}
		\nrm[L^2(\Sigma)]{u - u_\epsilon} \to 0, \quad
		\nrm[L^2(\Sigma)]{\nabla|_{\Sigma,x}(u - u_\epsilon)} \to 0, \quad \text{as~} \epsilon \to 0^+,
	\end{equation}
	where $\nabla|_{\Sigma,x}$ stands for the spacial component of the gradient on the manifold $\Sigma$.
	We aim to make the sequence $\{u_\epsilon\}_{\epsilon > 0}$ converge to $u$ in $H^1(\Sigma)$.
	According to \eqref{eq:ue1-LS2020}, it is left to show $\nrm[L^2(\Sigma)]{\partial_t(u - u_\epsilon)} \to 0$ as $\epsilon \to 0^+$.
	One can compute
	\begin{align*}
		\partial_t (u_\epsilon - \tilde u)(x,t)
		& = \frac 1 \epsilon \chi_0'(\frac {T-t} \epsilon) [\tilde u(x,t) - \tilde u(x,T)] + (1 - \chi_0(\frac {T-t} \epsilon)) \tilde u_t(x,t) - \tilde u_t(x,t) \\
		& = \frac 1 \epsilon \chi_0'(\frac {T-t} \epsilon) [(t-T) \tilde u_t(x,T) + \mathcal O(|t-T|^2)] - \chi_0(\frac {T-t} \epsilon) \tilde u_t(x,t) \\
		& = -\frac {T-t} \epsilon \chi_0'(\frac {T-t} \epsilon) [\tilde u_t(x,T) + \mathcal O(|t-T|)] - \chi_0(\frac {T-t} \epsilon) \tilde u_t(x,t).
	\end{align*}
	Note that $\frac {T-t} \epsilon \chi_0'(\frac {T-t} \epsilon) \to 0$ point-wise in $[0,T]$, and $\chi_0(\frac {T-t} \epsilon) \to 0$ point-wise in $[0,T)$ as $\epsilon \to 0^+$, so $\partial_t (u_\epsilon - \tilde u)(x,t) \to 0$ almost everywhere on $\overline \Sigma$, and hence
	\begin{equation} \label{eq:ue2-LS2020}
		\nrm[L^2(\Sigma)]{\partial_t(\tilde u - u_\epsilon)} \to 0, \quad \text{as~} \epsilon \to 0^+.
	\end{equation}
	Combining \eqref{eq:ue1-LS2020} with \eqref{eq:ue2-LS2020}, we arrive at
	\begin{equation} \label{eq:ue-LS2020}
		\nrm[H^1(\Sigma)]{\tilde u - u_\epsilon} \to 0, \quad \text{as~} \epsilon \to 0^+.
	\end{equation}

	The smoothness of $\tilde u$ implies $u_\epsilon \in C^K(\Sigma \cup \Sigma_-) \subset H^K(\Sigma \cup \Sigma_-)$.
%	\sq{The existence of the solution $v_\epsilon$ of \eqref{eq:vLe-LS2020} is guaranteed by [????????????????????????????????????????]}.
	For the system \eqref{eq:vLe-LS2020}, the prerequisites of Lemma \ref{lem:LLT210-LS2020} are all satisfied now, especially the compatibility requirements \eqref{eq:DPue-LS2020}, so we can conclude
	\[
	v_\epsilon \in C([-1,T]; H^K(\Omega)), \qquad \partial_t v_\epsilon \in C([-1,T]; H^{K-1}(\Omega)).
	\]
	By the Sobolev embedding theorem we know $H^K(\Omega) \subset C^2(\overline{\Omega})$ when $K - n/2 \geq 2$, so we set $K = \lceil n/2 \rceil + 2$, and thus
	\begin{equation} \label{eq:vEx1-LS2020}
		v_\epsilon \in C([-1,T]; C^2(\overline{\Omega})), \qquad \partial_t v_\epsilon \in C([-1,T]; C^1(\overline{\Omega})),
	\end{equation}
	so $\Delta v_\epsilon(\cdot,t) \in C(\ol{\Omega})$ for each $t$.
	Therefore, the equation $(\partial_t^2 -\Delta) v_\epsilon = 0$ implies
	\begin{equation} \label{eq:vEx2-LS2020}
		\partial_t^2 v_\epsilon(x,t) \text{~is continuous w.r.t.~} t.
	\end{equation}
	By \eqref{eq:vEx1-LS2020} and \eqref{eq:vEx2-LS2020} we obtain $v_\epsilon \in C^2(\overline Q)$, thus $v_\epsilon$ is well-defined on the slanted plane $\Gamma$ and $v_\epsilon |_\Gamma \in H^1(\Gamma)$.
	
	Finally, by Proposition \ref{prop:Bdd2-LS2020} and the convergence in \eqref{eq:ue-LS2020}, we see $\{ v_\epsilon |_\Gamma \}$ is a Cauchy sequence in $H^1(\Gamma)$,
	so we can define $v|_\Gamma$ as the limit of $v_\epsilon|_\Gamma$, i.e.,
	\[
	v|_\Gamma := \lim_{\epsilon \to 0^+} v_\epsilon |_\Gamma,
	\]
	and the estimate in Proposition \ref{prop:Bdd2-LS2020} implies $v|_\Gamma \in H^1(\Gamma)$.
	The proof is done.
\end{proof}

\begin{Remark} \label{rem:vEx-LS2020}
	According to \eqref{eq:nu1z-LS2020} and \eqref{eq:vLe-LS2020} we see that an $F \in H^{s_0}_{\ol{\Omega}}(\mathbb{R}^n)$ produces a $u$ and a $v_\epsilon$.
	By Proposition \ref{prop:sol_smooth_1} and Lemma \ref{lem:vEx-LS2020}, we can conclude there exists a large enough integer $s_0$ such that $u \in C^{\lceil n/2 \rceil + 2}$ and $v_\epsilon \in C^2$, respectively.
\end{Remark}

\subsection{The approximate inverse of \texorpdfstring{$A_0'$}{A0'}} \label{subsec:B-LS2020}

To introduce the approximate inverse $B$ of $A_0'$, we introduce some function spaces first.
Recall once more the notation given in \eqref{eq:nost-LS2020}.
We wish to consider functions $F$ in $\{ t = z \}$ that satisfy $ZF = 0$ outside $\Gamma$, since this is true in \eqref{eq:nu1z-LS2020-uF} when $\supp(f) \subset \ol{\Omega}$.
Moreover, for technical reasons we need the following weighted $H^1$-norm of $F$ on $\partial \Gamma$ to be finite:
\begin{equation} \label{eq:Weim-LS2020}
	\nrm[H_w^1(\partial \Gamma)]{F}
	:= \big( \int_{\partial \Gamma} \frac 1 {\sqrt{1-|y|^4}} |\nabla_{\partial \Gamma} F(y,z_y,z_y)|^2 \dif S \big)^{1/2}, \quad z_y := \sqrt{1-|y|^2}.
\end{equation}
Therefore, we consider functions $h$ on $\{ t = z \}$ that satisfy the conditions:
%\begin{equation} \label{eq:zc-LS2020}
%\supp(h) \subset \{ (y,z,z) \,;\, \abs{(y,z)} \leq 1, \text{ or } z \geq 0 \text{ and } \abs{y} \leq 1- \frac \sigma 2 \}, \quad Zh = 0 \text{ outside $\Gamma$}.
%\end{equation}
%\begin{equation} \label{eq:zc-LS2020}
%	\supp(h) \subset \{ (y,z,z) \,;\, \abs{(y,z)} \leq 1 \text{ and } \abs{y} \leq 1- \frac \sigma 2 \}, \ \text{~and~} \ Zh = 0 \text{ outside $\Gamma$}.
%\end{equation}
\begin{equation} \label{eq:zc-LS2020}
	\left\{\begin{aligned}
		& \supp(h) \subset \{ (y,z,z) \,;\, \text{``}\abs{(y,z)} \leq 1\text{''}, \text{ or } \text{``}z \geq 0 \text{ and } |y| \leq 1\text{''} \}, \\
		& Zh = 0 \text{~outside $\Gamma$}, \quad \text{and} \quad \nrm[H_w^1(\partial \Gamma)]{F} < +\infty.
	\end{aligned}\right.
\end{equation}
The integral requirement above is to make sure the integrals in \eqref{eq:ETX-LS2020}-\eqref{eq:ted1-LS2020} are finite.
%Later on, we will see that in \eqref{eq:ETX-LS2020}, we have to make sure that the denominator $1-|y|^2$ is positive away from zero, 
%and due to this technical reason, we have to put the prerequisite ``$\abs{y} \leq 1-\frac \sigma 2$'' in \eqref{eq:zc-LS2020}, and for the same reason we also need ``$\abs{x} \geq 1-\sigma$'' in \eqref{outside:ball}.
%These are due to computations in \eqref{eq:ETX-LS2020}-\eqref{eq:ted2-LS2020} and we do not explain the details here.

We define function spaces
\begin{equation} \label{eq:zc2-LS2020}
\left\{\begin{aligned}
\mathcal H' & := \{ h \in H^1_{\mathrm{loc}}(\{t = z \}) \,;\, h \text{~satisfies~} \eqref{eq:zc-LS2020} \}, \\
\mathcal H & := \{ h \in \mathcal H' \,;\, \supp h \subset \ol{\Gamma} \},
\end{aligned}\right.
\end{equation}
Recall the seminorm $\nrm[\mathcal H]{h}$ defined in \eqref{eq:NN0-LS2020},
\begin{equation*}
	\nrm[\mathcal H]{h} = \big( \frac 1 {\sqrt 2} \int_\Gamma (|\nabla_y h|^2 + |h_z + h_t|^2) \dif S \big)^{\frac 1 2}.
	%\Big( \int_{\Omega} |\nabla_x h(y,z,z)|^2 + |h_t|^2 \dif x \Big)^{1/2}.
\end{equation*}
Note that $Z = \p_z + \p_t$ is tangential to $\Gamma$, and hence $\nrm[\mathcal H]{h}$ corresponds to $\norm{\nabla_{\Gamma} h}_{L^2(\Gamma)}$.
Note also that by \eqref{eq:zc-LS2020} $\nrm[\mathcal H]{\,\cdot\,}$ is in fact a norm on $\mathcal H'$, so that $(\mathcal H', \nrm[\mathcal H]{\cdot})$ is a Banach space and $(\mathcal H, \nrm[\mathcal H]{\cdot})$ is a closed subspace.
Moreover, by Poincar\'e inequality, the norm $\norm{h}_{\mathcal H}$ is comparable to $\norm{h}_{H^1(\Gamma)}$ for $h \in \mathcal H$. We shall use these facts several times in the following computations.

By Lemma \ref{lem:vEx-LS2020} we can define an \emph{approximate inverse} $B$ of $A_0'$ as follows,
\begin{equation} \label{eq:B-LS2020}
	\left\{\begin{aligned}
	& B := \lim_{\epsilon \to 0^+} B_\epsilon, \ \text{where} \ B_\epsilon \colon C^K(\Sigma) \subset H^1(\Sigma) \to \mathcal H', \ \ u|_{\Sigma} \mapsto \tilde v_\epsilon, \\
	& \tilde v_\epsilon \text{~is the zero extension of~} v_\epsilon |_{\Gamma} \text{~to~} \{t = z\} \text{~where $v_\epsilon$ solves \eqref{eq:vLe-LS2020}}.
	\end{aligned}\right.
\end{equation}

%\begin{Remark} \label{rem:chi-LS2020}
%	Due to the prerequisite ``$\abs{y} \leq 1-\frac \sigma 2$'' in the definition of $\mathcal H'$ given in \eqref{eq:zc-LS2020}, the actual approximate inverse will be $\chi B$ where $\chi$ is a function satisfying
%	\begin{equation} \label{eq:chDe-LS2020}
%		\chi \in C^\infty(\{t = z\}), \quad |\chi| \leq 1, \quad \chi(y,z,z) = 
%		\left\{\begin{aligned}
%		1, & \quad \text{when~} |y| \leq 1 - \sigma, \\
%		0, & \quad \text{when~} |y| \geq 1 - \sigma/2.
%		\end{aligned} \right.
%	\end{equation}
%\end{Remark}

From the proof of Lemma \ref{lem:vEx-LS2020} we see $v|_{\Gamma} := \lim_{\epsilon \to 0^+} v_\epsilon |_{\Gamma} \in H^1(\Gamma)$, and $\tilde v$ is the zero extension of $v$, so $\tilde v \in \mathcal H'$. 
Note that the restriction of $u$ on $\{t = z\}$ belongs to $\mathcal H'$, so a solution of \eqref{eq:vL0-LS2020} will always have an extension in $\mathcal H'$.
The map $B$ shall be understood intuitively as a parametrix of $A_0'$, and $B A_0' F$ as an approximation of $F$ for any $F \in \mathcal H$ but not  for $F \in \mathcal H'$.
%This fact related to compact supports will be taken care of in Section \ref{sec:KK-LS2020}).
We will prove below in Proposition \ref{prop:Bdd3-LS2020} that in fact $B: H^1(\Sigma) \to \mathcal H'$ is a bounded map.

\subsection{Boundedness of the approximate inverse} %\label{sec:Bbd-LS2020}

\begin{Proposition} \label{prop:Bdd3-LS2020}
	The map $B$ extends as a bounded operator from $H^1(\Sigma)$ to $\mathcal H'$, and for $T \geq 1$, we have
	\begin{equation*}
	\nrm[\mathcal H]{Bu}
	\leq C e^{T/2} \nrm[H^1(\Sigma)]{u}
	\end{equation*}
	for some constant $C$ independent of $u$ and $T$.
\end{Proposition}

\begin{Remark}
    In Proposition \ref{prop:Bdd3-LS2020} the norm of $Bu$ is calculated on the light-like hyperplane $\Gamma$.
    Readers may also note that \cite{KM91} gives a similar result on time-like hyperplanes by using Carleman estimates.
\end{Remark}

\begin{proof}[Proof of Proposition \ref{prop:Bdd3-LS2020}]
    By density it is enough to prove the estimate when $u|_{\Sigma}$ is smooth. The equation \eqref{eq:vL0-LS2020} is a special case of \eqref{eq:vL-LS2020} when we choose $(\phi_0, \phi_1, G)$ according to
    \begin{equation*} %\label{eq:vLin-LS2020}
    	\left\{\begin{aligned}
    		\Delta \phi_0 & = 0 \text{~in~} \Omega, & \phi_0 & = u(\cdot,T) \text{~on~} \partial \Omega, \\
    		\phi_1 & = 0 \text{~in~} \Omega, & G & = 0 \text{~in~} Q.
    	\end{aligned}\right.
    \end{equation*}
	Recall the definition of $B$ in \eqref{eq:B-LS2020} and the proof of Lemma \ref{lem:vEx-LS2020}.
	From \eqref{eq:vnf-LS2020} we obtain
	\[
	\nrm[\mathcal H]{v_\epsilon}^2
	\leq C( \nrm[L^2(\Omega)]{\nabla \phi_0}^2 + ne^{T} \nrm[H^1(\Sigma)]{v_\epsilon}^2),
	\]
	and by taking the limit $\epsilon \to 0^+$ it gives
	\begin{align}
	\nrm[\mathcal H]{v}^2
	= \nrm[\mathcal H]{Bu}^2
%	& \leq C( \nrm[L^2(\Omega)]{\nabla \phi_0}^2 + ne^{T} \nrm[H^1(\Sigma)]{v}^2) \nonumber \\
	& = C( \nrm[L^2(\Omega)]{\nabla \phi_0}^2 + ne^{T} \nrm[H^1(\Sigma)]{u}^2).  \label{eq:unf-LS2020}
	\end{align}
    It remains to bound the $\nrm{\nabla \phi_0}$ term by $\nrm{u}$. Since $\phi_0$ is a harmonic function in $\Omega$ with Dirichlet data $u(\,\cdot\,, T)|_{\p \Omega}$, so it follows from standard estimates for the Dirichlet problem and from the trace theorem on $\Sigma$ that 
    \[
    \norm{\phi_0}_{H^1(\Omega)} \leq C \norm{u(\,\cdot\,, T)}_{H^{1/2}(\p \Omega)} \leq C \norm{u}_{H^1(\Sigma)}.
    \]
\begin{comment}
	For the trace operator $u \in H^1(\widetilde \Gamma_T) \mapsto u|_{\partial \Omega} \in H^{1/2}(\widetilde \Gamma_T \cap \Sigma)$, there exists a bounded right inverse \cite[\S 3.3.3]{tr83th}, denoted as $\mathcal T$, such that
	\[
	\mathcal T \colon \varphi \in H^{1/2}(\partial \Omega) \ \mapsto \ \mathcal T \varphi \in H^1(\Omega), \text{~with~} (\mathcal T \varphi)|_{\partial \Omega} = \varphi.
	\]
	We denote the operator norm of $\mathcal T$ as $\nrm{\mathcal T}$.
	For notational clearance we denote $f := u|_{\Gamma \cap \Sigma}$, then
	\begin{equation} \label{eq:Lb1-LS2020}
	\nrm[H^{1}(\Omega)]{\mathcal T f}
	\leq \nrm{\mathcal T} \nrm[H^{1/2}(\Gamma \cap \Sigma)]{u}.
	\end{equation}
	Now \eqref{eq:vLin-LS2020} is equivalent to
	\begin{equation*}
	\left\{\begin{aligned}
	-\Delta(\phi_0 - \mathcal T f) & = \Delta(\mathcal T f) && \text{~in~} \Omega, \\
	\phi_0 - \mathcal T f & = 0 && \text{~on~} \partial \Omega.
	\end{aligned}\right.
	\end{equation*}
	Standard elliptic estimates in PDE theory gives
	\begin{equation} \label{eq:Lb2-LS2020}
	\nrm[H^1(\Omega)]{\phi_0 - \mathcal T f}
	\leq \nrm[H^{-1}(\Omega)]{\Delta(\mathcal T f)}
	\lesssim \nrm[H^{1}(\Omega)]{\mathcal T f}.
	\end{equation}
	Combining the fact
	\(
	\nrm[H^1(\Omega)]{\phi_0 - \mathcal T f}
	\geq \nrm[H^1(\Omega)]{\phi_0} - \nrm[H^1(\Omega)]{\mathcal T f}
	\)
	with \eqref{eq:Lb1-LS2020} and \eqref{eq:Lb2-LS2020}, we obtain
	\begin{equation*} %\label{eq:Lb3-LS2020}
	\nrm[H^1(\Omega)]{\phi_0}
	\lesssim \nrm{\mathcal T} \nrm[H^{1/2}(\Gamma \cap \Sigma)]{u},
	\end{equation*}
\end{comment}
	Then \eqref{eq:unf-LS2020} becomes
	\begin{equation*} %\label{eq:Lb4-LS2020}
	\nrm[\mathcal H]{Bu}
	\leq C( \nrm[H^{1}(\Sigma)]{u} + \sqrt{n} e^{T/2} \nrm[H^1(\Sigma)]{u}).
	\end{equation*}
	This is the required statement.
\end{proof}

Readers may note that the constant $e^{T/2}$ in Lemma \ref{lem:tdlv-LS2020}, Proposition \ref{prop:Bdd2-LS2020} and Proposition \ref{prop:Bdd3-LS2020} might not be optimal.

\section{The error operator} \label{sec:KK-LS2020}

In this section we recover the function $F$ in \eqref{eq:nu1z-LS2020} by employing an iterative algorithm, and the function space at the beginning of the algorithm is different from these in the rest of the algorithm.
We define
\[
\mathcal H_{\mathrm{init}} := \{ h \in H^1_{\mathrm{loc}}(\{t = z \}) \,;\, h \text{~satisfies~} \eqref{eq:zc-LS2020} \} \, \cap \, H_{\overline \Gamma}^{s_0} (\{t = z\}).
\]
Here $s_0$ shall be chosen large enough such that Proposition \ref{prop:sol_smooth_1} can be applied.
The subscript in $\mathcal H_{\mathrm{init}}$ stands for ``initial'', and $\mathcal H_{\mathrm{init}}$ is the function space at the beginning of the iteration.
%The construction of $\chi$ in \eqref{eq:chDe-LS2020} guarantees
%\begin{equation} \label{eq:hch-LS2020}
%	\chi h = h \ \text{in} \ \mathcal H_{\mathrm{init}}.
%\end{equation}

Recall the operators $A_0'$ and $B_\epsilon$ defined in \eqref{eq:A1-LS2020} and \eqref{eq:B-LS2020}.
%Recall Remark \ref{rem:chi-LS2020}.
We wish to consider the operator $B A_0'$.
%where $\chi$ satisfies \eqref{eq:chDe-LS2020}.
%We need $\chi$ because $F$ is supported in $\{ \chi = 1 \}$ (recall \eqref{outside:ball}), so we want the approximate inverse to be also supported in $\{ \chi = 1 \}$.
By %\eqref{eq:hch-LS2020} and
%the same argument used in proving
Proposition \ref{prop:sol_smooth_1}, we see that if $h \in \mathcal H_{\mathrm{init}}$, we have $A_0' h \in H^1(\Sigma)$, and $h = A_0' h$ on $\partial \Gamma$.
%Here $A_0' \chi h$ stands for $A_0'(\chi h)$.
Moreover, by \eqref{eq:vLe-LS2020}, \eqref{eq:uep-LS2020} and \eqref{eq:B-LS2020} we see $A_0' h = B A_0' h$ on $\partial \Gamma$.
Hence, one can conclude
%$h = B A_0' \chi h$ on $\partial \Gamma$, and combining this with $h = \chi h$ on $\partial \Gamma$ induced by \eqref{eq:hch-LS2020}, one can conclude
\begin{equation} \label{eq:hcB1-LS2020}
	h - B A_0' h  = 0 \ \text{on} \ \partial \Gamma.
\end{equation}
By Proposition \ref{prop:Bdd3-LS2020} we have $B A_0' h \in \mathcal H'$, so,
\begin{equation} \label{eq:hcB2-LS2020}
	h - B A_0' h \in \mathcal H'.
\end{equation}
%We introduce the second cutoff to preserve the vanishing trace of $h - \chi B A_0' \chi h$.
Combining \eqref{eq:hcB1-LS2020} and \eqref{eq:hcB2-LS2020}, we can conclude
\begin{equation*} %\label{eq:hcB-LS2020}
	\forall h \in \mathcal H_{\mathrm{init}} \ \Rightarrow \
	h - B A_0' h \in \mathcal H.
\end{equation*}
Hence we can define an \emph{error operator} $K$ by:
\begin{equation} \label{eq:Kdef-LS2020}
K := \lim_{\epsilon \to 0^+} K_\epsilon, \quad \text{where} \quad
K_\epsilon \colon \mathcal H_{\mathrm{init}} \to \mathcal H, \  \ h \mapsto (I - B_\epsilon A_0') h.
\end{equation}

%There are two cutoff functions $\chi$ inside the operator $(I - \chi B A_0' \chi)$.
%The first $\chi$ (on the right-hand-side of $A_0'$) is to make sure the argument after \eqref{eq:zc-LS2020} stays valid.
%The second $\chi$ (on the left-hand-side of $B$ is to guarantee $(I - \chi B A_0' \chi)$ maps $\mathcal H_{\mathrm{init}}$ into not only $\mathcal H'$ but also, more importantly, into $\mathcal H$;
%without the second $\chi$, this will not be true.

In the linearized inverse problem for \eqref{eq:A1-LS2020}, $K F$ indicates the difference between the original $F$ and the approximation $B(u|_\Sigma) = B A_0' F$.
%For $F \in \mathcal{H}' \cap H^{s_0}(\{t=z\})$ with $\supp F \subset \supp \chi$ \footnote{\color{blue} Is the same cutt-off introduced in \eqref{eq:uep-LS2020}, $\chi_0?$},
For $F \in \mathcal{H}' \cap H^{s_0}(\{t=z\})$ with $\supp F \subset \supp \overline \Gamma $,
we have
\[
(I - K)F = B(A_0' F) = B(A_0' F) = B(u|_{\Sigma}).
\]
Note that $u|_{\Sigma}$, $B$ and $K$ are all known.
Hence one would expect to recover $F$ using the formula $(I - K)^{-1} B(u|_{\Sigma})$ provided that $I - K$ is invertible, which would hold true in particular if the operator norm $\nrm{K}$ between suitable spaces is strictly less than $1$.

Instead of \eqref{eq:Kdef-LS2020}, it would be more convenient if the domain and image of $K$ are the same.
This is necessary, for instance, to perform a Neumann series argument.
This leads us to considering the restriction of $K$ to $\mathcal H$.
We denote the restriction as $\tilde K$, i.e.,
\begin{equation*} %\label{eq:Ktde-LS2020}
\tilde K \colon \mathcal{H} \cap H^{s_0}(\{t=z\}) \to  \mathcal H, \ \ h \mapsto \ (I - B A_0') h.
\end{equation*}
We will prove in Proposition \ref{prop:Kc-LS2020} that $\tilde K$ extends as a bounded operator $\mathcal{H} \to \mathcal{H}$ with norm strictly less than $1$ if $T$ is large enough.
Based on this, a reconstruction formula
\(
F = (I + \sum_{j \geq 0} \tilde K^j K) B (u|_{\Sigma})
\)
is given in \eqref{eq:FA-LS2020}.
After obtaining $F$, the $f$ can be recovered by $f = ZF$.
Based on this, a stability result is also given, see Proposition \ref{prop:RF-LS2020} for details.
An illustration of the recovering procedure of $F$ is given in Fig.~\ref{fig:1-LS2020}.
\begin{figure}[h]
	\begin{tikzcd}
		\mathcal H_{\mathrm{init}}
		\arrow[r,"B A_0'"] &
		\mathcal H',
	\end{tikzcd} \
	\begin{tikzcd}
		\mathcal H_{\mathrm{init}}
		\arrow[r,"K"] &
		\mathcal H \arrow[r,"\tilde K"] &
		\mathcal H \arrow[r,"\tilde K"] &
		\cdots \arrow[r,"\tilde K"] &
		\mathcal H \arrow[r,"\tilde K"] &
		\cdots
	\end{tikzcd}
	\caption{An illustration of the recovering procedure.}
	\label{fig:1-LS2020}
\end{figure}

To elaborate how the error operator $K$ works, we give an intuitive example.
Assume $F \in \mathcal H'$, $u$ and $v$ satisfy the following systems,
\begin{equation*}
\left\{\begin{aligned}
	(\partial_t^2 -\Delta) u & = 0 \quad \text{in~} \{ t > z\} \\ 
	u & = F \ \ \text{on~} \{ t = z\}
\end{aligned}\right.
,\quad
\left\{\begin{aligned}
	(\partial_t^2 -\Delta) v & = 0 && \text{in}\; Q \\ 
	v & = u && \text{on } \Sigma \cup \Sigma_- \\
	v = \phi_0, \ v_t & = 0 && \text{in}\; \Gamma_T
\end{aligned}\right.
\end{equation*}
with
\begin{equation*}
	\left\{\begin{aligned}
	\Delta \phi_0 & = 0 && \text{in~} \Omega, \\
	\phi_0 & = u(\cdot,T) && \text{on~} \partial \Omega.
	\end{aligned}\right.
\end{equation*}
Then, by ignoring the regularity issues we could say $KF = (u - v)|_\Gamma$.

\subsection{The energy estimates} %\label{subsec:PeHy-LS2020}

For $\tau \in [0,T]$ and $h \in H^1(\Rn \times \{t = \tau\})$, we define the energy $E(\tau,h)$ as follows
\begin{equation*} 
	E(\tau, h)
	:= \int_{\Rn} (|\nabla_x h(x,\tau)|^2 + |h_t(x,\tau)|^2) \dif x.
\end{equation*}
We also define a functional $\nrm[\mathcal H_T]{h}$ as follows,
\begin{equation} \label{eq:HTh-LS2020}
	\nrm[\mathcal H_T]{h}
	:= \big( \int_{\Omega} (|\nabla_x h(x,T)|^2 + |h_t(x,T)|^2) \dif x \big)^{1/2}.
\end{equation}
The relations among $\nrm[\mathcal H]{\cdot}$, $\nrm[\mathcal H_T]{\cdot}$ and $E(T, \cdot)$ are described in the following lemmas.

\begin{Lemma} \label{lem:ne-LS2020}
	Assume $u$ is $C^2(\overline Q)$ and $u$ and $v$ are associated to each other by means of \eqref{eq:vL0-LS2020},
	and $u_\epsilon$ and $v_\epsilon$ are defined as in the proof of Lemma \ref{lem:vEx-LS2020}. Set  $w_\epsilon:= u - v_\epsilon$.
	Then we have
	\[
	\lim_{\epsilon \to 0^+} \nrm[\mathcal H]{w_\epsilon}
	= \lim_{\epsilon \to 0^+} \nrm[\mathcal H_T]{w_\epsilon}.
	\]
	Moreover, we have
	\[
	\nrm[\mathcal H_T]{u}^2 = \nrm[\mathcal H_T]{u - \phi_0}^2 + \int_\Omega |\nabla \phi_0(x)|^2 \dif x.
	\]
\end{Lemma}

\begin{proof}
%	We first analyze the case where $h \in C^2$.
	
	According to the construction of $v_\epsilon$, we know $\omega_\epsilon \in C^2(\overline Q)$.
	Moreover, $\omega_\epsilon$ satisfies
	\begin{equation*}
		\left\{\begin{aligned}
			(\partial_t^2 -\Delta) \omega_\epsilon & = 0 && \text{in}\; Q, \\ 
			v & = u - u_\epsilon && \text{on } \Sigma \cup \Sigma_-.
		\end{aligned}\right.
	\end{equation*}
	Then according to \eqref{eq:0g0-LS2020} we have
	\[
	0
	= \nrm[\mathcal H_T]{\omega_\epsilon}^2 - \nrm[\mathcal H]{\omega_\epsilon}^2 - 2 \Re \int_{\widetilde \Sigma_T}  \partial_t \overline{\omega_\epsilon} \partial_\nu \omega_\epsilon \dif S,
	\]
	Note that $\omega_\epsilon \in C^2(\overline Q)$.
	Combining this with Lemma \ref{lem:tdlv-LS2020}, we can compute
	\begin{align}
	\nrm[\mathcal H]{\omega_\epsilon}^2
	& \leq \nrm[\mathcal H_T]{\omega_\epsilon}^2 + 2 \nrm[H^1(\Sigma)]{\omega_\epsilon} \nrm[L^2(\Sigma)]{\partial_\nu \omega_\epsilon} \nonumber \\
	& \leq \nrm[\mathcal H_T]{\omega_\epsilon}^2 + C \nrm[H^1(\Sigma)]{\omega_\epsilon} (\nrm[\mathcal H]{\omega_\epsilon} + C \nrm[H^1(\Sigma)]{\omega_\epsilon}) \nonumber \\
	& = \nrm[\mathcal H_T]{\omega_\epsilon}^2 + C \nrm[H^1(\Sigma)]{\omega_\epsilon} \nrm[\mathcal H]{\omega_\epsilon} + C^2 \nrm[H^1(\Sigma)]{\omega_\epsilon}^2 \nonumber \\
	& = \nrm[\mathcal H_T]{\omega_\epsilon}^2 + 2C^2 \nrm[H^1(\Sigma)]{\omega_\epsilon}^2 + \frac 1 2 \nrm[\mathcal H]{\omega_\epsilon}^2 + C \nrm[H^1(\Sigma)]{\omega_\epsilon}^2.
	\label{eq:we2-LS2020}
	\end{align}
	We have $\nrm[H^1(\Sigma)]{\omega_\epsilon} \to 0$ as $\epsilon \to 0^+$ due to \eqref{eq:ue-LS2020}.
	And the function $\omega_\epsilon~(:= u - v_\epsilon)$ is independent of $\epsilon$ on $\Gamma_T$ because $v_\epsilon$ is independent of $\epsilon$ on $\Gamma_T$, see \eqref{eq:uep-LS2020}.
	Hence, \eqref{eq:we2-LS2020} implies $\nrm[\mathcal H]{\omega_\epsilon}$ is uniformly bounded when $\epsilon \to 0^+$.
	Based on this fact, we re-compute \eqref{eq:we2-LS2020} in the following way with the help of Lemma \ref{lem:tdlv-LS2020},	
	\begin{align}
	|\nrm[\mathcal H]{\omega_\epsilon}^2 - \nrm[\mathcal H_T]{\omega_\epsilon}^2|
	\leq 2 \nrm[H^1(\Sigma)]{\omega_\epsilon} \nrm[L^2(\Sigma)]{\partial_\nu \omega_\epsilon}
	\leq C \nrm[H^1(\Sigma)]{\omega_\epsilon} (\nrm[\mathcal H]{\omega_\epsilon} + C \nrm[H^1(\Sigma)]{\omega_\epsilon}). \label{eq:we21-LS2020}
	\end{align}
	By the boundedness of $\nrm[\mathcal H]{\omega_\epsilon}$ w.r.t.~$\epsilon$ as well as the fact that $\nrm[H^1(\Sigma)]{\omega_\epsilon} \to 0$ as $\epsilon \to 0^+$, we can conclude from \eqref{eq:we21-LS2020} that
	\begin{align*}
	\nrm[\mathcal H]{\omega_\epsilon} - \nrm[\mathcal H_T]{\omega_\epsilon}  \to 0 \quad \text{as} \quad \epsilon \to 0^+.
	\end{align*}

	For the second claim, we have
	\begin{align*}
	\nrm[\mathcal H_T]{u - \phi_0}^2
	& = \int_\Omega |\nabla_x \big( u(x,T) - \phi_0 \big)|^2 \dif x + \int_\Omega |u_t|^2 \dif x \\
	& = \nrm[\mathcal H_T]{u}^2 - 2 \Re \int_\Omega \nabla_x \big( u \big) \cdot \nabla_x \overline{\phi_0} \dif x + \int_\Omega |\nabla_x \phi_0|^2 \dif x \\
	& = \nrm[\mathcal H_T]{u}^2 - 2 \Re \int_\Omega \nabla_x \big( u - \phi_0 \big) \cdot \nabla_x \overline{\phi_0} \dif x - \int_\Omega |\nabla_x \phi_0|^2 \dif x \\
	& = \nrm[\mathcal H_T]{u}^2 - 2 \Re \int_{\partial \Gamma_T} (u - \phi_0) \partial_\nu \overline{\phi_0} \dif x + 2 \Re \int_\Omega (u - \phi_0) \Delta \phi_0 \dif x \\
	& \quad - \int_\Omega |\nabla_x \phi_0|^2 \dif x \\
	& = \nrm[\mathcal H_T]{u}^2 - \int_\Omega |\nabla_x \phi_0|^2 \dif x.
	\end{align*}
	The last equal sign is due to $u - \phi_0 = 0$ on $\partial \Gamma_T$ and $\Delta \phi_0 = 0$.
	The proof is complete.
\end{proof}

Recall the weighted norm given in \eqref{eq:Weim-LS2020}.

\begin{Lemma} \label{lem:EuT0-LS2020}
	Assume $h$ satisfies $(\p_t^2-\Delta) h = 0$ in $\{t \geq z\}$ with the restriction $h |_\Gamma$ satisfying $h|_\Gamma \in \mathcal H_{\mathrm{init}}$, then we have
	\begin{equation*}
		\nrm[\mathcal H_T]{h}^2 \leq E(T,h) \leq
		\left\{\begin{aligned}
			& \nrm[\mathcal H]{h}^2 + T \nrm[H_w^1(\partial \Gamma)]{h}^2, \text{~if~} h \in \mathcal H_{\mathrm{init}}, \\
			& \nrm[\mathcal H]{h}^2, \text{~if~} h \in \mathcal H \cap H_{\overline \Gamma}^{s_0}(\{t = z\}).
		\end{aligned}\right.
	\end{equation*}
%	\begin{equation*}
%	\nrm[\mathcal H_T]{h}^2 \leq E(T,h) \leq
%	\left\{\begin{aligned}
%		& \nrm[\mathcal H]{h}^2 + \frac {\sqrt{2} T} {\sqrt \sigma} \nrm[H^1(\partial \Gamma)]{h}^2, \text{~if~} h \in \mathcal H_{\mathrm{init}}, \\
%		& \nrm[\mathcal H]{h}^2, \text{~if~} h \in \mathcal H \cap H_{\overline \Gamma}^{s_0}(\{t = z\}).
%	\end{aligned}\right.
%	\end{equation*}
\end{Lemma}

\begin{proof}
	By the definition of the energy and the norm, it is obvious that $\nrm[\mathcal H_T]{h}^2 \leq E(T,h)$.
	
	Now we prove the second inequality.
	Recall Remark \ref{rem:vEx-LS2020}.
	Replacing the function $v$ in identity \eqref{eq:rts-LS2020} by $h$ and integrating the identity in the region
	\[
	\{ (y,z,t) \in \R^{n+1} \,;\, y \in \R^{n-1},\, z \leq t \leq T \},
	\]
	we can obtain
	\begin{equation} \label{eq:ETh-LS2020}
	E(T,h) = \frac 1 {\sqrt 2} \int_{\{t = z\}} (|\nabla_y h|^2 + |\partial_z \big( h(y,z,z) \big)|^2) \dif S.
	\end{equation}
	The computation is similar to what has been done in Section \ref{subsec:IBVP-LS2020}, so we omit the details here.
	
	When $h|_\Gamma \in \mathcal H' \cap H_{\overline \Gamma}^{s_0}(\{t = z\})$, we know $h \in C^2(\{t \geq z\})$.
	Moreover, $h \in \mathcal H'$ gives  $\partial_z \big( h(y,z,z) \big) = 0$ outside $\Gamma$ and $\nabla_y h$ is constant along the curve $\gamma(z) := (y,z,z)$ for each $y$,
	so \eqref{eq:ETh-LS2020} becomes
	\begin{align}
	E(T,h)
	& = \frac 1 {\sqrt 2} \int_{\{t = z \,;\, t \leq T\}} (|\nabla_y h|^2 + |\partial_z \big( h(y,z,z) \big)|^2) \dif S \nonumber \\
	& =  \nrm[\mathcal H]{h}^2 + \frac 1 {\sqrt 2} \int_{\{t = z \,;\, t \leq T\} \backslash \Gamma} |\nabla_y \big( h(y, z_y, z_y) \big)|^2 \dif S \quad (\partial_z \big( h(y,z,z) \big) = 0) \nonumber \\
	& = \nrm[\mathcal H]{h}^2 + \int_{\pi(\{t = z \,;\, t \leq T\} \backslash \Gamma)} \big| \nabla_y \big( h(y, z_y, z_y) \big) \big|^2 \dif y \dif z \nonumber \\
	& \leq \nrm[\mathcal H]{h}^2 + T \int_{|y| \leq 1} |\nabla_y \big( h(y, z_y, z_y) \big)|^2 \dif y \nonumber \\
	& = \nrm[\mathcal H]{h}^2 + T \int_{|y| \leq 1} \frac {1} {1-|y|^2} |\vec Xh(y, z_y, z_y)|^2 \dif y, \label{eq:ETX-LS2020}
	\end{align}
	where $z_y := \sqrt{1-|y|^2}$, $\pi \colon (x,t) \mapsto x$ is a projection map, and $\vec X$ is a vector of vector fields defined as
	\[
	\vec X := \sqrt{1-|y|^2} \nabla_y - y \partial_z - y \partial_t.
	\]

	The vector field $\vec X$ is tangential to $\partial \Gamma_+ := \partial \Gamma \cap \{(y,z,t) \,;\, z \geq 0 \}$.
	To see this, we denote $S_1 := \{(y,z,t) \,;\, |(y,z)|^2 = 1 \}$ and $S_2 := \{(y,z,t) \,;\, t = z \}$.
	Then $\partial \Gamma_+ = S_1 \cap S_2$, and $\vec X (|(y,z)|^2 - 1) = \vec X(t - z) = 0$ on $\partial \Gamma_+$.
	Therefore, $\vec X$ is tangential to both $S_1$ and $S_2$ at $\partial \Gamma_+$, and hence tangential to $\partial \Gamma_+$.
	
	Denote the volume form on $\partial \Gamma_+$ as $\df S$, then it can be checked that
	\[
	\frac {\sqrt{1 + |y|^2}} {\sqrt{1 - |y|^2}} \dif y = \dif S.
	\]
	Hence, \eqref{eq:ETX-LS2020} implies
	\begin{equation} \label{eq:ted1-LS2020}
	E(T,h)
	\leq \nrm[\mathcal H]{h}^2 + T \int_{\partial \Gamma_+} \frac 1 {\sqrt{1-|y|^4}} |\vec Xh|^2 \dif S
	\leq \nrm[\mathcal H]{h}^2 + T \nrm[H_w^1(\partial \Gamma)]{h}^2.
	\end{equation}
%	Recall that $\supp h|_\Gamma \subset \{ (y,z,t) \,;\, |y| \leq 1 - \sigma/2 \}$ for the fixed constant $\sigma \in (0,1)$, so $\vec X h = 0$ when $|y| \geq 1 - \sigma/2$, see also \cite[Proof of Lemma 2.9]{MePoSa20}.
%	Therefore,
%	\begin{align}
%	E(T,h)
%	& \leq \nrm[\mathcal H]{h}^2 + \frac {T} {\sqrt{1-(1-\sigma/2)^4}} \int_{\partial \Gamma_+} |\vec Xh|^2 \dif S  \label{eq:ted2-LS2020} \\
%	& \leq \nrm[\mathcal H]{h}^2 + \frac {\sqrt{2} T} {\sqrt \sigma} \nrm[H^1(\partial \Gamma)]{h}^2. \nonumber
%	\end{align}
	Hence, when $h \in C^2 \cap \mathcal H$, we know $h = 0$ on $\partial \Gamma$, so $E(T,h) \leq \nrm[\mathcal H]{h}^2$.
	The proof is done.
\end{proof}

\subsection{Boundedness of the error operators} %\label{subsec:BKK-LS2020}

In the definition of $K$ and $\tilde K$, the two operators are defined on certain subsets of $\mathcal H'$ and $\mathcal H$, respectively.
As we mentioned earlier, their domain of definitions can be extended to the whole $\mathcal H'$ and $\mathcal H$.
To that end, we first show some boundedness result of $K$ and $\tilde K$ in certain dense subsets of their domain.
For simplicity we denote the operator norm of $\tilde K$ as $\nrm{\tilde K}$, i.e.,~$\nrm{\tilde K} = \nrm[\mathcal H \to \mathcal H]{\tilde K}$, and we also write 
\[
\nrm{K} := \sup \Big\{ \frac {\nrm[\mathcal H]{Kh}} {\nrm[\mathcal H]{h} + \nrm[H_w^1(\partial \Gamma)]{h}} \,;\, h \in \mathcal H_{\mathrm{init}}, \, h\neq 0 \Big\}.
\]

\begin{Lemma} \label{lem:Kb-LS2020}
	Let $\sigma \in (0,1)$ be as in \eqref{outside:ball}.
	If $C>0$ is the constant from Lemma \ref{lem:EuT0-LS2020} and $s_0$ is determined as in Remark \ref{rem:vEx-LS2020}, then we have
	\begin{alignat}{2}
	\nrm[\mathcal{H}]{Kh} & \leq \nrm[\mathcal{H}]{h} + T^{1/2} \nrm[H_w^1(\partial \Gamma)]{h}, & \quad & h \in \mathcal H_{\mathrm{init}}, \label{eq:Ewue2-LS2020} \\
	\nrm[\mathcal{H}]{\tilde Kh} & \leq \nrm[\mathcal{H}]{h}, && h \in \mathcal{H} \cap H_{\overline \Gamma}^{s_0}(\{t = z\}). \label{eq:Ewue-LS2020}
	\end{alignat}
%	Note that \eqref{eq:Ewue2-LS2020} implies $\nrm{K} \leq C T^{1/2}/\sigma^{1/4}$.
\end{Lemma}

\begin{proof}
	For any function $h \in \mathcal H' \cap H_{\overline \Gamma}^{s_0}(\{t = z\})$, let $u$ be the solution of system \eqref{eq:nu1z-LS2020} driven by $h$ so that $u = h$ on $\Gamma$,
	then according to Proposition \ref{prop:sol_smooth_1} we can conclude $u \in C^{\lceil n/2 \rceil + 2}$ in $\{t \geq z\}$.
	
	Assume $v$ is associated to $u$ by \eqref{eq:vL0-LS2020},
	and $u_\epsilon$ and $v_\epsilon$ are defined as in the proof of Lemma \ref{lem:vEx-LS2020},
	and we denote $w_\epsilon = u - v_\epsilon$.
%	By \eqref{eq:hch-LS2020} we know $\chi h = h$.
	By \eqref{eq:Kdef-LS2020} we see
	\[
	K_\epsilon h
	= w_\epsilon |_\Gamma
	= h - v_\epsilon |_\Gamma
	= \omega_\epsilon,
	\]
	so we have
	\begin{equation*}
	\nrm[\mathcal H]{Kh}^2
	= \lim_{\epsilon \to 0 ^+} \nrm[\mathcal H]{K_\epsilon h}^2
	= \lim_{\epsilon \to 0 ^+} \nrm[\mathcal H]{w_\epsilon}^2
	= \lim_{\epsilon \to 0 ^+} \nrm[\mathcal{H}_T]{w_\epsilon}^2,
	\end{equation*}
	where the last equal sign is due to Lemma \ref{lem:ne-LS2020}.
	Note that on $\Gamma_{T}$, $w_\epsilon = u - \phi_0$, which is independent of $\epsilon$, so $\nrm[\mathcal H]{Kh}^2
	\leq \nrm[\mathcal H_T]{u - \phi_0}^2$.
	Combining this with Lemma \ref{lem:ne-LS2020} and Lemma \ref{lem:EuT0-LS2020} we have
	\begin{alignat}{2}
	\nrm[\mathcal H]{Kh}^2
	& \leq \nrm[\mathcal H_T]{u}^2 & \quad & \text{(by Lemma \ref{lem:ne-LS2020})} \nonumber \\
	& \leq E(T,u)
	\leq \nrm[\mathcal H]{u}^2 + T \nrm[H_w^1(\partial \Gamma)]{u}^2 && \text{(by Lemma \ref{lem:EuT0-LS2020})} \nonumber \\
	& = \nrm[\mathcal H]{h}^2 + T \nrm[H_w^1(\partial \Gamma)]{h}^2. \nonumber
	%\label{eq:Ewu-LS2020}
	\end{alignat}
	This is \eqref{eq:Ewue2-LS2020}.
%	and we obtain $\nrm{K} \leq \mathcal C$ from \eqref{eq:Ewue2-LS2020}.
	When $h \in \mathcal H \cap H_{\overline \Gamma}^{s_0}(\{t = z\})$, we know $h |_{\partial \Gamma}= 0$, so similar to \eqref{eq:Ewue2-LS2020} we have
	\begin{equation} \label{eq:Ewu2-LS2020}
	\nrm[\mathcal H]{\tilde K h}^2
	\leq \nrm[\mathcal H_T]{u}^2
	\leq E(T,u)
	\leq \nrm[\mathcal H]{u}^2 
	= \nrm[\mathcal H]{h}^2.
	\end{equation}
	This gives \eqref{eq:Ewue-LS2020}.
	We arrive at the conclusion.
\end{proof}

We are ready to extend $\tilde K$ and to show its boundedness with norm strictly less than \nolinebreak[4] $1$.

\begin{Proposition} \label{prop:Kc-LS2020}
	$\tilde K$ can be extended to $\mathcal{H}$, and it is a bounded linear operator $\mathcal{H} \to \mathcal{H}$ with $\nrm{\tilde K} < 1$.
\end{Proposition}

\begin{proof}
The linearity of $\tilde K$ can be easily seen from systems \eqref{eq:nu1z-LS2020} and \eqref{eq:vL-LS2020}.
Recall the $s_0$ in Remark \ref{rem:vEx-LS2020}.
It can be checked that $\mathcal{H} \cap H_{\overline \Gamma}^{s_0}(\{t = z\})$ is dense in $\mathcal{H}$.
Assume $h \in \mathcal H$, we choose a sequence $\{h_j\}_{j \geq 1}$ in $\mathcal{H} \cap H_{\overline \Gamma}^{s_0}(\{t = z\})$ such that $h_j \to h$ under the $\nrm[\mathcal H]{\cdot}$-norm.
By \eqref{eq:Ewue-LS2020} we know $\{ \tilde Kh_j \}_{j \geq 1}$ is a Cauchy sequence in $\mathcal H$, so $\lim_{j \to \infty} \tilde Kh_j$ exists and the limit is unique, and we define
\begin{equation}\label{k_tilde_z}
\tilde Kh := \lim_{j \to \infty} \tilde Kh_j.
\end{equation}
%It can be checked that $\mathcal{H} \cap H_{\overline \Gamma}^{s_0}(\{t = z\})$ is dense in $\mathcal H$, so by the extension and \eqref{eq:Ewue-LS2020} we see that $\nrm{\tilde H} \leq 1$.

The norm estimate for $\nrm{\tilde K}$ in Lemma \ref{lem:Kb-LS2020} can be improved to $\nrm{\tilde K} < 1$ by using local energy decay.
%Assume $T_0$ is fixed such that $a(\tau) = 0$ when $\tau \geq T_0$ (recall \eqref{eq:aC-LS2020}), namely, the hyperplane $\mathcal R^n(\tau)$ becomes horizontal when $\tau \geq T_0$ (see Fig.~\ref{fig:hyf-LS2020}).
Indeed, fix a constant $T_0 \in (1,T)$, then when $t > T_0$, \eqref{eq:nu1z-LS2020} gives
\begin{equation} \label{eq:nu2z-LS2020}
	\left\{\begin{aligned}
	(\partial_t^2 -\Delta) u & = 0 && \text{in~} \Rn \times [T_0, T], \\ 
	u = u_0, \ u_t & = u_1 && \text{on~} \{ t = T_0\},
	\end{aligned}\right.
\end{equation}
where $u_0(x) = u(x,T_0)$ and $u_1(x) = u_t(x,T_0)$.
By the local energy result in \cite{ike2005var} we see that when $t - T_0$ is large enough, we have
\begin{equation} \label{eq:leday-LS2020}
	\nrm[{\mathcal H}_{t}]{u}^2 \leq \frac {\lambda E(T_0,u)} {t - C}, \ \forall t \colon \max \{T_0,C\} \leq t \leq T,
\end{equation}
for some real numbers $\lambda$ and $C$ independent of $u$ and $t$.
Here the $\mathcal H_t$-norm is defined by changing $T$ to $t$ in \eqref{eq:HTh-LS2020}.
The authors remind that originally the result in \cite{ike2005var} is for wave equations in an exterior domain with a bounded Dirichlet obstacle contained inside.
But the method in \cite{ike2005var} applies directly to the case where the obstacle is an empty set, i.e.~applies to the free upper space case, e.g.~\eqref{eq:nu2z-LS2020}.
See also \cites{tam1981potential, Vodev04}.

Now let $h \in \mathcal{H} \cap H_{\overline \Gamma}^{s_0}(\{t = z\})$ and let $u$ be associated with $h$ by \eqref{eq:nu1z-LS2020}.
By \eqref{eq:leday-LS2020} and Lemma \ref{lem:EuT0-LS2020}, when $T - T_0$ is large enough, we see that
\[
C \nrm[{\mathcal H}_T]{u}^2
\leq \lambda' E(T_0,u)
\leq \lambda' \nrm[\mathcal H]{u}^2
\]
of some constant $\lambda' < 1$, and the second equal sign is due to the tact that $h \in \mathcal H$.
Hence, the chain of inequalities \eqref{eq:Ewu2-LS2020} can be further improved: for $h \in \mathcal{H} \cap H_{\overline \Gamma}^{s_0}(\{t = z\})$ we have
\begin{equation*}
	\nrm[\mathcal H]{\tilde Kh}^2
	\leq C \nrm[\mathcal H_T]{u}^2
	\leq \lambda' \nrm[\mathcal H]{u}^2
	= \lambda' \nrm[\mathcal H]{h}^2.
\end{equation*}
Therefore, we can conclude $\nrm{\tilde K} < 1$.
The proof is complete.
\end{proof}

\begin{Remark} %\label{rem:Kc-LS2020}
	We believe that the arguments in \cite[Proof of Theorem 1]{SU09} in proving $\nrm{\tilde K} < 1$ may not directly apply to our case.
	This is mainly because ${\tilde K}$ is defined by a limit process, see \eqref{k_tilde_z}. In the limit, the condition $\nrm{\tilde K} < 1$ becomes $\nrm{\tilde K} \leq 1$.
	Consequently, it is not possible to use a Neumann series argument to invert an operator of the form $I-\tilde K$.
	See Section \ref{sec:ReLS-LS2020} for details.
\end{Remark}

\subsection{Reconstruction and stability result of the linearized sound speed} \label{sec:ReLS-LS2020}

By definition \eqref{eq:Kdef-LS2020} we have $(I - K)F = B A_0' F$.
%When $\supp F \subset \{ \chi  = 1 \}$, we have $(I - K)F = \chi B A_0' F$.
Proposition \ref{prop:Kc-LS2020} says the operator norm of $\tilde K$ is less than $1$, so we can use Neumann series to recover $F$.
Before that, we need the following claim.
\begin{Claim} \label{clm:At-LS2020}
	$(B A_0') K = K (B A_0').$
\end{Claim}

\begin{proof}
	We have
	\begin{align*}
	(B A_0') K
	& = (B A_0') (I - B A_0')
	= B A_0' - (B A_0') (B A_0') \\
	& = (I - B A_0') (B A_0')
	= K (B A_0').
	\end{align*}
	The proof is done.
\end{proof}

\begin{Proposition} \label{prop:RF-LS2020}
	Assume $F \in \mathcal H_{\mathrm{init}}$, and $f$ satisfies $ZF = f$. Let $u$ be the solution of the system \eqref{eq:nu1z-LS2020}.
	Then 
	\begin{equation} \label{eq:RF1-LS2020}
	F = (I + \sum_{j \geq 0} \tilde K^j K) B( u|_\Sigma), \quad
	f = Z[(I + \sum_{j \geq 0} \tilde K^j K) B( u|_\Sigma)],
	\end{equation}
	Moreover, we have the stability
	\begin{equation} \label{eq:RF2-LS2020}
	\nrm[L^2(\Omega)]{f}
	\leq \mathcal C (\nrm[H^1(\Sigma)]{A_0 f} + \frac {T^{1/2}} {\sigma^{1/4}} \norm{A_0 f}_{H^1(\partial \Gamma)} ),
	\end{equation}
	where $\mathcal C = C \sqrt{nT} e^{CT/2} [1 + \nrm{K} (1 - \nrm{\tilde K})^{-1}]$ and the constant $C$ is independent of $f$, $K$ and $T$.
\end{Proposition}

\begin{Remark}
	A straightforward computation shows that Proposition \ref{prop:fA1f-LS2020} is an immediate consequence of Proposition \ref{prop:RF-LS2020}.
\end{Remark}

\begin{Remark}
	Note that $\nrm[\mathcal H]{\tilde K^j K Bu} \leq \nrm{\tilde K}^j \nrm{K} \nrm[\mathcal H]{Bu} \lesssim \nrm{\tilde K}^j \nrm{K} \nrm[H^1(\Sigma)]{u}$ with $0 < \nrm{\tilde K} < 1$ by Proposition \ref{prop:Kc-LS2020}.
	The last inequality sign is due to Proposition \ref{prop:Bdd3-LS2020}.
	Hence the infinite sum in \eqref{eq:RF1-LS2020} converges.
\end{Remark}

\begin{proof}[Proof of Proposition \ref{prop:RF-LS2020}]
	Note that \eqref{eq:Kdef-LS2020} gives the identity $I - K = B A_0'$, and $\tilde K$ is a restriction of $K$, so $I - \tilde K = B A_0'$ also holds when restricted to $\mathcal H$.
	
	For $F \in \mathcal H_{\mathrm{init}}$, we denote $F_1 := (I - B A_0') F = KF$,
	thus $F_1 \in \mathcal H$ and $(I - \tilde K) F_1 = B A_0' F_1$ because $I - \tilde K = B A_0'$ on $\mathcal H$,
	so
	\[
	F_1 = (I - \tilde K)^{-1} B A_0' F_1 = (I - \tilde K)^{-1} B A_0' K F,
	\]
	and by Claim \ref{clm:At-LS2020} we obtain $F_1 = (I - \tilde K)^{-1} K B A_0' F$.
	Because $F_1 = (I - B A_0') F$, we have $F - B A_0' F = (I - \tilde K)^{-1} K B A_0' F$, which finally gives
	\begin{align}
	F
	& = [I + (I - \tilde K)^{-1} K] B A_0' F
	= [I + (I - \tilde K)^{-1} K] B A_0' F \nonumber \\
	& = (I + \sum_{j \geq 0} \tilde K^j K) B( u|_\Sigma). \label{eq:FA-LS2020}
	\end{align}
%	We used the fact $\chi F = F$.
	Note that $K$, $\tilde K$, $B$ and $u|_\Sigma$ are known, thus $F$ can be recovered using \eqref{eq:FA-LS2020}.
	The invertibility of $(I - \tilde K)$ is guaranteed by Proposition \ref{prop:Kc-LS2020}.
	The image of $K$, i.e.~$\mathcal H$, falls into the domain of $(I - \tilde K)^{-1}$, so the RHS of \eqref{eq:FA-LS2020} is well-defined. The identity \eqref{eq:FA-LS2020} is an analogue of \cite[eq.~(11)]{SU09}.
	Recall $f = ZF$, so finally we obtained \eqref{eq:RF1-LS2020}.
	
	Combining Proposition \ref{prop:Bdd3-LS2020}, Lemma \ref{lem:Kb-LS2020} and the fact $\nrm{\tilde K} < 1$, we have
	\begin{align*}
	\nrm[\mathcal H]{F}
	& = \nrm[\mathcal H]{(I + (I - \tilde K)^{-1} K) Bu}
	\leq \nrm[\mathcal H]{Bu} + (1 - \nrm{\tilde K})^{-1} \nrm[\mathcal H]{K Bu} \\
	& \leq C \nrm[\mathcal H]{Bu} + C (1 - \nrm{\tilde K})^{-1} \nrm{K} (\nrm[\mathcal H]{Bu} + \frac {T^{1/2}} {\sigma^{1/4}} \nrm[H^1(\partial \Gamma)]{Bu}) \\
	& \leq C \big[ \nrm[\mathcal H]{Bu} + (1 - \nrm{\tilde K})^{-1} \nrm{K} (\nrm[\mathcal H]{Bu} + \frac {T^{1/2}} {\sigma^{1/4}} \nrm[H^1(\partial \Gamma)]{u}) \big] \\
	& \leq C [1 + \nrm{K} (1 - \nrm{\tilde K})^{-1}] ( \nrm[\mathcal H]{Bu} + \frac {T^{1/2}} {\sigma^{1/4}} \nrm[H^1(\partial \Gamma)]{u} ) \\
	& \leq C \sqrt{n} e^{T/2} [1 + \nrm{K} (1 - \nrm{\tilde K})^{-1}] ( \nrm[H^1(\Sigma)]{u} + \frac {T^{1/2}} {\sigma^{1/4}} \nrm[H^1(\partial \Gamma)]{u} ) .
	%
%	& \leq C [1 + \nrm{K} (1 - \nrm{\tilde K})^{-1}] (\nrm[\mathcal H]{Bu} + \norm{u}_{H^1(\partial \Sigma)} ) \\
%	& \leq C \sqrt{n} e^{T/2} [1 + \nrm{K} (1 - \nrm{\tilde K})^{-1}] (\nrm[H^1(\Sigma)]{u} + \norm{u}_{H^1(\partial \Sigma)} )
	\end{align*}
	On the other side, the norm of $F$ satisfies (see \eqref{eq:NN0-LS2020})
	\begin{align*}
	\nrm[\mathcal H]{F}^2
	& = \int_{\Omega} |\nabla_y F(y,z,z)|^2 + |Z F(y,z,z)|^2 \dif x
	\geq \int_{\Omega} |Z F(y,z,z)|^2 \dif x \\
	& = \int_{\Omega} |f(y,z)|^2 \dif x.
	\end{align*}
	Hence,
	\begin{equation*}
	\nrm[L^2(\Omega)]{f} \leq C \sqrt{n} e^{T/2} [1 + \nrm{K} (1 - \nrm{\tilde K})^{-1}] (\nrm[H^1(\Sigma)]{u} + \frac {T^{1/2}} {\sigma^{1/4}} \norm{u}_{H^1(\partial \Gamma)} ).
	\end{equation*}
	Recall $u |_\Sigma = A_0 f$.
	The proof is done.
\end{proof}

\begin{Remark} %\label{rem:Bext-LS2020}
	It is worth noting that, when the dimension is odd, and when the final time $T$ is large enough, by the Huygens' principle we know $u(\cdot, T)$ will be zero in $\Omega$ (see e.g.~\cite[\S 2.~eq.~(31)]{evans2010pde}),
	hence the final condition in \eqref{eq:vL0-LS2020} will be $v = 0$ and $v_t = 0$ in $\Gamma_T$, and this matches with the values of $u$ and $u_t$ in $\Gamma_T$ because they are all zero functions at time $T$.
	In this case, the function $v$ will be exactly the same as $u$, and hence $\tilde K$ will be a zero map,
	and so the reconstruction formula \eqref{eq:RF1-LS2020} can be simplified as	
	\begin{equation*}
		F = (I + K) B( u|_\Sigma), \quad f = Z[(I + K) B( u|_\Sigma)].
	\end{equation*}
	However, in even dimensions we do not have such a conclusion.
\end{Remark}

\appendix

\section{Well-posedness of the linear forward problem} \label{appendix:stationary}

We wish to consider the well-posedness for the wave equation in Sobolev spaces that have a negative smoothness index with respect to time. For smooth coefficients such results may be found in \cite[Chapter 9]{hormander1976linear}. Since we need to deal with coefficients having finite regularity, we will give the required results and proofs in this Appendix.

Throughout this Appendix, the notation $H^s$ with $s\in \mathbb{R}$ stands for $H^s(\mathbb{R}^n)$. For $s\in \R$, consider the norm $\norm{f}_{H^s} = \norm{J^s f}_{L^2}$ where $J^s = (1-\Delta_x)^{s/2}$, and the corresponding inner product as $(f, g)_{H^s} = (J^s f, J^s g)_{L^2}$. We will consider wave equations of the form 
\[
\p_t^2 u -\gamma^{-1} \p_j(\gamma a^{jk} \p_k u) + qu = 0
\]
where $\gamma$ and $(a^{jk})$ are positive. This form includes both the standard wave equation $(\p_t^2 - c(x)^{2} \Delta_x) u = 0$ as well as the Riemannian wave equation $(\p_t^2 - \Delta_g) u = 0$. However, in the proofs we need to use the modified inner product $(\gamma u, v)_{L^2}$ in order to make the elliptic part symmetric.

We begin with an energy estimate for smooth functions.

\begin{Lemma} \label{lemma_wp_first}
Consider the operator 
\[
Au = -\gamma^{-1} \p_j(\gamma a^{jk} \p_k u) + qu
\]
where $(a^{jk}(x))$ is a symmetric matrix and one has for some $M \geq 1$ 
\begin{equation} \label{wave_ellipticity}
M^{-1} \abs{\xi}^2 \leq a^{jk}(x) \xi_j \xi_k \leq M \abs{\xi}^2, \qquad M^{-1} \leq \gamma(x) \leq M \qquad \text{a.e. in} \ \R^n.
\end{equation}
Let $s \in \R$, and assume further the Sobolev multiplier properties 
\begin{equation} \label{sobolev_multiplier}
\norm{[J^s, \gamma^{-1} \p_j(\gamma a^{jk} \p_k)]}_{H^{s+1} \to L^2} + \norm{q}_{H^{s+1} \to H^s} \leq M.
\end{equation}
Let $T > 0$ and let $t_0 \in [0,T]$. Then for any $u \in C^{\infty}_c(\R^n \times \R)$ one has 
\begin{equation*}
\sum_{j=0}^1 \norm{\p_t^j u}_{L^{\infty}((0,T), H^{s+1-j})} \leq C (\norm{(\p_t^2 + A) u}_{L^1((0,T), H^{s})} + \norm{u(t_0)}_{H^{s+1}} + \norm{\p_t u(t_0)}_{H^s}),
\end{equation*}
where $C = C(M, T)>0$.
\end{Lemma}
\begin{proof}
Write $F = (\p_t^2 + A)u$ and $v = J^s u$, and define the energy 
\begin{align*}
E(t) &= \norm{u(t)}_{H^{s}}^2 + \norm{\nabla_x u(t)}_{H^s}^2 + \norm{\p_t u(t)}_{H^s}^2 \\
 &= \norm{v(t)}_{L^2}^2 + \norm{\nabla_x v(t)}_{L^2}^2 + \norm{\p_t v(t)}_{L^2}^2.
\end{align*}
Thus 
\[
M^{-2} E(t) \leq E_1(t) := (\gamma v(t), v(t))_{L^2} +  (\gamma a^{jk} \p_j v(t), \p_k v(t))_{L^2} + (\gamma \p_t v(t), \p_t v(t))_{L^2}.
\]
Since $(\p_t^2 + A) u = F$, note that the function $v$ solves the equation 
\[
\p_t^2 v - \gamma^{-1} \p_j(\gamma a^{jk} \p_k v) = G,
\]
where 
\[
G = J^s F + [J^s, \gamma^{-1} \p_j(\gamma a^{jk} \p_k)] u - J^s(qu).
\]
Differentiating $E_1(t)$ and using the equation $\p_t^2 v - \gamma^{-1} \p_j(\gamma a^{jk} \p_k v) = G$ gives 
\begin{align*}
\frac{1}{2} E_1'(t) &= (\gamma \p_t v(t), v(t))_{L^2} + (\gamma a^{jk} \p_j v(t), \p_k \p_t v(t))_{L^2} + (\gamma \p_t^2 v(t), \p_t v(t))_{L^2} \\
&= (\gamma \p_t v(t), v(t))_{L^2} + (\gamma G(t), \p_t v(t))_{L^2}.
\end{align*}
Note that 
\[
\norm{\gamma^{1/2} G(t)}_{L^2} \leq M^{1/2} (\norm{F(t)}_{H^s} + M \norm{u(t)}_{H^{s+1}}).
\]
Thus we have 
\begin{align*}
\frac{1}{2} E_1'(t) &\leq \norm{\gamma^{1/2} \p_t v(t)}_{L^2}( \norm{\gamma^{1/2} v(t)}_{L^2} + M^{1/2} \norm{F(t)}_{H^s} + M^{3/2} \sqrt{E(t)} ) \\
 &\leq \sqrt{E_1(t)} ((1 + M^{5/2}) \sqrt{E_1(t)} + M^{1/2} \norm{F(t)}_{H^s}).
\end{align*}
Since $\sqrt{E_1(t)}$ may not be smooth, we study $\sqrt{\eps + E_1(t)}$ for $\eps > 0$. This satisfies 
\[
\frac{d}{dt} \sqrt{\eps + E_1(t)} = \frac{E_1'(t)}{2 \sqrt{\eps + E_1(t)}} \leq C_1 \sqrt{\eps + E_1(t)} + M^{1/2} \norm{F(t)}_{H^s}
\]
where $C_1 = 1 + M^{5/2}$. By Gronwall's inequality we obtain for any $\tau \in [0, T]$ that 
\[
\sqrt{\eps + E_1(\tau)} \leq e^{C_1 T} (\sqrt{\eps + E_1(t_0)} + M^{1/2} \int_{t_0}^{\tau} \norm{F(t)}_{H^s} e^{-C_1 t} \,dt).
\]
Letting $\eps \to 0$ proves the required statement.
\end{proof}

The following result will be used for obtaining sufficient conditions for \eqref{sobolev_multiplier}.

\begin{Lemma} \label{lemma_sobolev_multiplier}
For any $a \in H^{\alpha_0}$ with $\alpha_0 > n/2$ one has the Sobolev multiplier property 
\[
\norm{au}_{H^{\alpha}} \leq C_{\alpha_0} \norm{a}_{H^{\alpha_0}} \norm{u}_{H^{\alpha}}, \qquad \abs{\alpha} \leq \alpha_0.
\]
For any $a \in H^{s_0}$ with $s_0 > n/2 + 1$ one has the commutator estimate 
\[
\norm{J^s(au) - a J^s u}_{L^2} \leq C_{s_0} \norm{\nabla a}_{H^{s_0-1}} \norm{u}_{H^{s-1}}, \qquad 1 \leq \abs{s} \leq s_0.
\]
\end{Lemma}
\begin{proof}
The first part follows e.g.\ from \cite[Section 13.10]{Taylor3}. The second part for $s=1$ follows from Calder\'on's commutator formula \cite[Section VII.3.5]{Stein}, and for $n/2+1 < s \leq s_0$ it follows from the Kato-Ponce inequality \cite{KatoPonce}. Since we could not locate a precise reference for this result under the given conditions, we will give a proof following \cite{Kato1975}.

The commutator estimate follows if we show that the operator $T = [J^s, a] J^{1-s}$ is bounded on $L^2$. First let $1 \leq s \leq s_0$. Computing the Fourier transform of $Tu$ gives 
\begin{align*}
\widehat{Tu}(\xi) &= \br{\xi}^s (\hat{a} \ast \br{\xi}^{1-s} \hat{u}) - \hat{a} \ast \br{\xi} \hat{u} \\
 &= \int (\br{\xi}^s - \br{\eta}^s) \hat{a}(\xi-\eta) \br{\eta}^{1-s} \hat{u}(\eta) \,d\eta.
\end{align*}
Since the function $f(t) = (1+t^2)^{s/2}$ satisfies $f''(t) \geq 0$ for $s \geq 1$, we have 
\[
\abs{\br{\xi}^s - \br{\eta}^s} \leq s (\br{\xi}^{s-1} + \br{\eta}^{s-1}) \abs{\xi-\eta}.
\]
%\tbl{[MS: what kind of inequality holds if $0 \leq s \leq 1$?]}
Let $h$ be the function satisfying $\hat{h}(\xi) = \abs{\xi} \,\abs{\hat{a}(\xi)}$. It follows that $T = T_1 + T_2$ where 
\begin{align*}
\abs{\widehat{T_1 u}(\xi)} &\leq s \br{\xi}^{s-1} \int \hat{h}(\xi-\eta) \br{\eta}^{1-s} \abs{\hat{u}(\eta)} \,d\eta, \\
\abs{\widehat{T_2 u}(\xi)} &\leq s \int \hat{h}(\xi-\eta) \abs{\hat{u}(\eta)} \,d\eta.
\end{align*}
Now one has 
\begin{align*}
\norm{\widehat{T_1 u}}_{L^2} &\leq s \norm{h \mathscr{F}^{-1}\{ \br{\eta}^{1-s} \abs{\hat{u}(\eta)} \} }_{H^{s-1}} \leq s \norm{h}_{H^{s-1} \to H^{s-1}} \norm{\hat{u}}_{L^2}, \\
\norm{\widehat{T_2 u}}_{L^2} &\leq s \norm{h}_{L^{\infty}} \norm{\mathscr{F}^{-1}\{ \abs{\hat{u}(\eta)} \}}_{L^2} \leq s \norm{h}_{L^{\infty}} \norm{\hat{u}}_{L^2}.
\end{align*}
Moreover, since $s_0 > n/2+1$, we have the estimate $\norm{h}_{L^{\infty}} \leq C \norm{h}_{H^{s_0-1}} = C \norm{\nabla a}_{H^{s-1}}$ and $\norm{h}_{H^{s-1} \to H^{s-1}} \leq C \norm{h}_{H^{s_0-1}} = C \norm{\nabla a}_{H^{s_0-1}}$. The required commutator estimate for $1 \leq s \leq s_0$ follows by using the Plancherel theorem.

Now let $-s_0 \leq s \leq -1$. We use duality and compute 
\begin{align*}
\norm{[J^s, a] u}_{L^2} &= \sup_{\norm{v}_{L^2}=1} ([J^s, a] u, v) \\
 &= \sup_{\norm{v}_{L^2}=1} (J^{s-1} u, J^{1-s} a J^s v - J^1 a v) \\
 &\leq  \sup_{\norm{v}_{L^2}=1} \norm{u}_{H^{s-1}} \norm{J^1 [J^{-s}, a] J^s v}_{L^2}.
\end{align*}
Note that 
\[
\norm{J^1 [J^{-s}, a] J^s v}_{L^2} \leq \norm{[J^{-s}, a] J^s v}_{L^2} + \norm{[J^{-s}, \nabla a] J^s v}_{L^2} + \norm{[J^{-s}, a] \nabla J^s v}_{L^2}.
\]
We use the commutator estimate for $[J^{-s}, a]$ to conclude that the first and last terms on the right are $\lesssim \norm{v}_{L^2}$. For the middle term we use that 
\[
\norm{(J^{-s}(\nabla a) - (\nabla a) J^{-s}) J^s v}_{L^2} \leq \norm{(\nabla a) J^s v}_{H^{-s}} + \norm{(\nabla a) v}_{L^2} \lesssim \norm{v}_{L^2}.
\]
This concludes the proof.
\end{proof}

We next give a simple solvability result which follows directly from the energy estimate in Lemma \ref{lemma_wp_first} and duality. In fact, under suitable regularity assumptions for the coefficients one has a unique solution in the class $C([0,T], H^{s+1}) \cap C^1([0,T], H^s)$, see e.g.\ \cite{Smith1998}.

\begin{Lemma} %\label{lemma_wp_second}
Assume \eqref{wave_ellipticity} and  \eqref{sobolev_multiplier} for some $s_0 > n/2$, and let $-s_0 \leq s \leq s_0$. Given any $u_0 \in H^{s+1}$, $u_1 \in H^{s}$ and $F \in L^1((0,T), H^{s})$, there is $u \in L^{\infty}((0,T), H^{s+1})$ which is a weak solution of 
\begin{equation*} %\label{wave_cauchy_general}
(\p_t^2 + A) u = F \text{ in $\R^n \times (0,T)$}, \qquad u(0) = u_0, \qquad \p_t u(0) = u_1,
\end{equation*}
in the sense that 
\[
(\gamma u, (\p_t^2 + A) \varphi) = (\gamma F, \varphi) - (\gamma u_0, \p_t \varphi(0)) + (\gamma u_1, \varphi(0))
\]
for any $\varphi \in C^{\infty}_c(\R^n \times [0,T))$. This solution satisfies
\[
\norm{u}_{L^{\infty}((0,T), H^{s+1})} \leq C (\norm{F}_{L^1((0,T), H^{s})} + \norm{u_0}_{H^{s+1}} + \norm{u_1}_{H^s})
\]
where $C = C(n, s, M, E, T)$.
\end{Lemma}

\begin{proof}
Let $X = C^{\infty}_c(\R^n \times [0,T))$, i.e.\ any $\varphi \in X$ vanishes near $t=T$. Now if $\varphi_1, \varphi_2 \in X$ and $(\p_t^2 + A)\varphi_1 = (\p_t^2 + A)\varphi_2$, then Lemma \ref{lemma_wp_first} with $t_0=T$ implies that $\varphi_1 = \varphi_2$. We may thus define the linear functional 
\[
\ell: (\p_t^2 + A) X \to \R, \ \ \ell((\p_t^2 + A)\varphi) = (\gamma F, \varphi) - (\gamma u_0, \p_t \varphi(0)) + (\gamma u_1, \varphi(0)).
\]
Applying Lemma \ref{lemma_wp_first} with $t_0=T$ again, one has 
\[
\abs{\ell((\p_t^2 + A)\varphi)} \lesssim (\norm{F}_{L^1 H^{s}} + \norm{u_0}_{H^{s+1}} + \norm{u_1}_{H^{s}}) \norm{(\p_t^2 + A)\varphi}_{L^1 H^{-s-1}}.
\]
By the Hahn-Banach theorem $\ell$ extends as a bounded linear functional on $L^1 H^{-s-1}$, and by duality it can be represented by a function $u \in L^{\infty} H^{s+1}$ satisfying 
\[
\ell((\p_t^2+A)\varphi) = (\gamma u, (\p_t^2+A)\varphi), \qquad \varphi \in X,
\]
and 
\[
\norm{u}_{L^{\infty} H^{s+1}} \lesssim \norm{F}_{L^1 H^{s}} + \norm{u_0}_{H^{s+1}} + \norm{u_1}_{H^{s}}.
\]
This proves the result.
\end{proof}

We next move to solutions that are in negative Sobolev spaces with respect to time. For any integer $k \geq 0$ and any $\alpha \in \R$, we let $H^{-k}((-T,T), H^{\alpha}(\R^n))$ be the dual of $H^k_0((-T,T), H^{-\alpha}(\R^n))$.

\begin{Lemma}\label{eq:forward_problem_z}
Assume \eqref{wave_ellipticity} and \eqref{sobolev_multiplier} for some $s_0 > n/2 + 2$. Additionally, assume that  $\gamma, a^{jk}\in C^{1,1}(\R^n)$ for $j,k=1,2, \ldots,n$.
Let $k \geq 0$ be an integer, and let $-s_0 \leq s \leq s_0$. Given any $F \in H^{-k}((-T,T),H^{s})$ with $F|_{\{t < 0\}} = 0$, there is a unique distributional solution $u \in H^{-k}((-T,T), H^{s+1})$  of 
\begin{equation} \label{wp_vanishing_negative_time}
(\p_t^2+A) u = F \text{ in $\R^n \times (-T,T)$}, \qquad u|_{\{t < 0\}} = 0.
\end{equation}
This solution satisfies 
\[
\norm{u}_{H^{-k}((-T,T), H^{s+1})} \leq C \norm{F}_{H^{-k}((-T,T), H^{s})}
\]
where $C = C(n, k, s, M, E, T)$.
\end{Lemma}
\begin{proof}
Consider the space $X = C^{\infty}_c(\R^n \times [0,T))$ and the linear functional 
\[
\ell: (\p_t^2+A)X \to \R, \ \ \ell((\p_t^2+A)\varphi) = (\gamma F, E\varphi)
\]
where $E\varphi$ is any smooth extension of $\varphi$ from $\R^n \times [0,T)$ to $\R^n \times (-T,T)$. The right hand side is well defined since $F|_{\{ t < 0 \}} = 0$. Choosing $E$ to be a bounded extension operator on $H^k$, one has 
\[
\abs{\ell((\p_t^2+A)\varphi)} \lesssim \norm{F}_{H^{-k}((-T,T), H^s)} \norm{\varphi}_{H^k((0,T), H^{-s})}.
\]
Applying Lemma \ref{lemma_wp_first} with $t_0=T$ to $\varphi, \p_t \varphi, \ldots, \p_t^k \varphi$ gives that 
\[
\abs{\ell((\p_t^2+A)\varphi)} \lesssim \norm{F}_{H^{-k}((-T,T), H^s)} \norm{(\p_t^2+A)\varphi}_{H^k((0,T), H^{-s-1})}.
\]
Using the Hahn-Banach theorem and duality, there is $u \in H^{-k}((-T,T), H^{s+1})$ satisfying 
\[
(\gamma u, (\p_t^2+A)\varphi) = \ell((\p_t^2 + A)\varphi) = (\gamma F, E \varphi), \qquad \varphi \in X,
\]
and 
\[
\norm{u}_{H^{-k}((-T,T), H^{s+1})} \lesssim \norm{F}_{H^{-k}((-T,T), H^s)}.
\]
Given any $\psi \in C^{\infty}_c(\R^n \times (-T,0))$ we may find a solution of $(\p_t^2 + A)\varphi = \psi$ with $\varphi(0) = \p_t \varphi(0) = 0$. Then we have $(\gamma u, \psi) = (\gamma F, E \varphi)$ for any extension of $\varphi$, and choosing the zero extension gives $u = 0$ for $t < 0$. Thus $u$ is a distributional solution of \eqref{wp_vanishing_negative_time}.

It remains to prove that any $u \in H^{-k}((-T,T), H^{s+1})$ satisfying 
\[
(\p_t^2+A) u = 0 \text{ in $\R^n \times (-T,T)$}, \qquad u|_{\{t < 0\}} = 0,
\]
must be identically zero. Fix some $F \in C^{\infty}_c(\R^n \times (-T,T))$ and choose $\varphi$ to be a solution of 
\begin{equation*} %\label{eq:smooth_coeff_z}
(\p_t^2 + A)\varphi = F \text{ in $\R^n \times (-T,T)$}, \qquad \varphi(T) = \p_t \varphi(T) = 0.
\end{equation*}
Then $\varphi \in L^{\infty} H^{s_0+1}$.
For any $\chi \in C^{\infty}([-T,T])$ with $\chi = 1$ near $[0,T]$ and $\chi = 0$ near $t=-T$ one has 
\[
(u, F) = (u, (\p_t^2 + A)\varphi) = (u, (\p_t^2 + A)(\chi(t) \varphi)).
\]
By hypothesis, we have $\gamma, a^{jk}\in C^{1,1}(\R^n)$ for $j,k=1,2, \ldots,n$. By \cite[Theorem 4.7]{Smith1998}, we conclude that $\varphi$ belongs to $L^2 H^{s_0+3}$. We claim that $\varphi$ has more regularity in time. Indeed, since $\partial_t^2\varphi = F- A\varphi$, we deduce $\varphi\in H^2 H^{s_0+1}$. We can iterate this argument so that by means of the Morrey's embedding one has $\varphi \in C^2((-T, T); \R^n)$ and hence also $\chi(t)\varphi$.
Since $\varphi(T)=\partial_t\varphi(T)=0$ and $\chi(-T)= \chi^\prime(-T)$=0, one can thus integrate by parts and obtain that 
\[
(u, F) = ((\p_t^2 + A) u, \chi(t)\varphi) = 0.
\]
Since $F$ was arbitrary we obtain $u=0$ as required.
\end{proof}

As a consequence, we prove the well-posedness of \eqref{eq:1-LS2020} and \eqref{eq:2-LS2020}.
	
	\subsubsection*{\bf Proof of Proposition \ref{prop:well_posedness_z_1}} We start by setting $\mathcal{U}:= \mathcal{V}+ H_1(t-z)$. Thus, $\mathcal{U}$ is a solution to \eqref{eq:2-LS2020} if and only if $\mathcal{V}$ is a solution to the IVP
\begin{equation} \label{initial_eq_z_z_z_y}
(\partial_t^2 - \eta^{-1} \Delta)  \mathcal{V}=-\eta^{-1} (1- \eta) \delta(t-z) \ \ \text{in} \ \ \R^{n+1}, \qquad  \mathcal{V}|_{\left\{t<-1\right\}} = 0.
\end{equation}
To solve \eqref{initial_eq_z_z_z_y}, we aim to use  Lemma \ref{eq:forward_problem_z} with the following configuration
\[
\gamma=\eta, \quad a^{jk}= \eta^{-1} \delta_{jk}, \quad s=0, \quad k=-1 \quad \text{and} \quad s_0=\left[n/2\right] +
3.
\]
Indeed, let us verify one by one the conditions from Lemma \ref{eq:forward_problem_z}:
\begin{itemize}
\item[-] {\bf Regularity}. By hypothesis, $\eta\in C^{1,1}(\mathbb{R}^n)$. Then, $\gamma$ and $a^{jk}$, $j,k=1, 2, \ldots,n$, defined above also belong to $C^{1,1}(\mathbb{R}^n)$.
\item[-] {\bf Condition \eqref{wave_ellipticity}}. It easily follows by \eqref{outside:ball_z}.
\item[-] {\bf Condition \eqref{sobolev_multiplier}}. It is clear that $s_0>n/2+2$. Set $\widetilde{\eta}:= \eta^{-1}-1\in H^{s_0}(\mathbb{R}^n)$.  Let $s\in \mathbb{R}$ so that $1 \leq \abs{s} \leq s_0$. Let $\varphi\in \mathcal{S}(\mathbb{R}^n)$. Using the fact that $[J^s, \Delta]\equiv 0$, we get
\begin{align*}
\norm{[J^s, \gamma^{-1} \p_j(\gamma a^{jk} \p_k)] \varphi}_{L^2}&= \norm{[J^s, \eta^{-1}\Delta]\varphi}_{L^2}= \norm{[J^s, \widetilde{\eta}\,\Delta]\varphi}_{L^2}\\
& = \norm{J^s\left( \widetilde{\eta} \, \Delta \varphi\right)- \widetilde{\eta}\, J^s(\Delta \varphi)}_{L^2}\\
&  \leq C_{s_0} \norm{\nabla \, \widetilde{\eta}}_{H^{s_0-1}} \norm{\Delta \varphi}_{H^{s-1}}\\
&  \leq C_{s_0} \norm{\nabla \, \widetilde{\eta}}_{H^{s_0-1}} \norm{\varphi}_{H^{s+1}}.
\end{align*}
In the last inequality, we used the second inequality from Lemma \ref{lemma_sobolev_multiplier} with $a=\widetilde{\eta}$ and $u=\Delta \varphi$. Hence, condition \eqref{sobolev_multiplier} is satisfied for  $1 \leq \abs{s} \leq s_0$. The case $s=0$ holds trivially. 
\end{itemize}
On the other hand, by \eqref{outside:ball}, we deduce that the source $\eta^{-1} (1-\eta)\delta(t-z)$ in \eqref{initial_eq_z_z_z_y} belongs to  $H^{-1}((-T, T); L^2(\R^n))$. Hence Lemma \ref{eq:forward_problem_z} ensures that there exists a unique distributional solution
$\mathcal{V}\in  H^{-1}((-T, T); H^1(\R^n))$ to \eqref{initial_eq_z_z_z_y} with
\begin{equation}\label{ine:mistake}
\norm{\mathcal{V}}_{H^{-1}((-T, T); H^1(\R^n))}
\lesssim \norm{\eta^{-1}(1- \eta)\delta(t-z) }_{H^{-1}((-T, T); L^2(\R^n))} \lesssim C(T,M).
\end{equation}
In addition, one can verify that $\mathcal{U}:= \mathcal{V}+ H_1(t-z)$ is a distributional solution to the IVP \eqref{eq:2-LS2020}. Since $H_1(t-z)\in H^{-1}((-T, T); H^1_{loc}(\R^n))$, we deduce $\mathcal{U}\in H^{-1}((-T, T); H^1_{loc}(\R^n)) $. Using \eqref{ine:mistake} we immediately get estimate \eqref{estimate_local_z} for any smooth and compactly supported function $\chi$ in the spatial variable.

A straightforward computation shows that $\partial_t^2 \s\mathcal{U}$ satisfies \eqref{eq:1-LS2020}. By uniqueness of distributional solutions we deduce $\tilde{\mathcal{U}}=\partial_t^2\s\mathcal{U}$. By fixing $\chi\in C^\infty_0(\mathbb{R}^n)$ so that $\chi(x)=1$ when $|x|<2$ and $\chi(x)=0$ when $|x|>3$, and thanks to the trace theorem one has the following sequence of inequalities
\[
\norm{\tilde{\mathcal{U}}|_{\Sigma_T}}_{H^{-3}((-T, T); H^{1/2}({\mathbb{S}}^{n-1}))}  \lesssim \norm{\chi\,\mathcal{U}}_{H^{-1}((-T, T); H^1(\R^n))} \lesssim  C(T,M, \norm{\chi}_{C^1}).
\]
This finishes the proof of Proposition \ref{prop:well_posedness_z_1}.

\subsubsection*{\bf Proof of Proposition \ref{prop:sol_smooth_1}} 
We start by proving the uniqueness. It is reduced to proving that zero is the unique distributional solution to the homogeneous equation
\begin{equation*} 
(\partial_t^2 -\Delta)\, U =0 \; \text{in}\; \R^{n+1}, \quad U \, |_{\left\{t<-1\right\}}=0,
\end{equation*}
which is true by \cite[Theorem 23.2.7]{hormander1983linear}. The method we shall use to prove the existence is the so-called progressing wave expansion method, see e.g. \cite[Lemma 1]{MR778095} and \cite[Theorem 1]{RakeshUhlmann}. For any $j\geq 0$, define
\begin{equation*} %\label{s_j_z_z}
s_+^j = \left\{\begin{matrix}
 s^j , &s\geq 0, \\ 
 0 , &s<0.
\end{matrix}\right.
\end{equation*}
Note that $s^0_+=H(s)$ is the unidimensional Heaviside function at $s\in \R$. Let $f\in H^{s_0}_{\ol{\Omega}}(\mathbb{R}^n)$. Assume that $f\in C^m_c(\mathbb{R}^n)$, where $m\in \mathbb{N}$ (wich will be fixed later) with $s_0> n/2+m$. This is always possible due to Morrey's inequality and the fact that $\supp(f)\subset \ol{\Omega}$. Consider $N\in\mathbb{N}$ and suppose that the solutions to  $(\partial_t^2 -\Delta) U =-f(x)\delta(t-z)$ have the following ansatz
\begin{equation}\label{pr:wave_exp}
U (x,t)= \sum_{j=0}^N a_j (x)(t- z)_+^j + R_N(x,t),
\end{equation}
where the  coefficients $a_j$, $j=0,1, \ldots, N$, and the remainder term $R_N$ satisfy the initial value conditions
\[
 R_N|_{t<-1}=0, \quad  a_{j}|_{z<-1}=0, \quad j=0, 1, \dots, N.
\]
A straightforward computation shows that the remainder term $R_N$ must hold
\begin{equation}\label{id_recursive_formulae}
\begin{aligned}
(\partial_t^2- \Delta)R_N(x,t)&= - \left(2\, \p_z \,a_0+f\right)\delta(t-z)  +\left( \Delta \, a_N\right)(t-z)^N_+\\
&\quad  \quad - \sum_{j=0}^{N-1}\left(2(j+1) \partial_z \, a_{j+1}-\Delta \, a_j \right)(t-z)^j_+.
\end{aligned}
\end{equation}
The task now is proving the existence of the coefficients $(a_j)_{j=0}^N$ and $R_N$, satisfying the recursive identity \eqref{id_recursive_formulae}. One expects to get a smoother remainder term $R_N$ when $N$ is large enough. This can be done by dropping most of the non-smooth terms on the right of \eqref{id_recursive_formulae}. By standard ODE techniques, it is easy to see that if
\begin{equation}\label{a_0:z}
a_0(y,z)=-\frac{1}{2}\int_{-\infty}^0 f (y, s+z)\, ds,
\end{equation}
and for $k=0,1, \dots,  N-1$: 
\begin{equation}\label{a_0:z_z}
a_{k+1}(y,z)= \frac{1}{2(k+1)}\int_{-\infty}^0  (\Delta \, a_k)(y, s+z)\,ds,
\end{equation}
then the remainder term $R_N$ has to satisfy
\begin{equation}\label{eq:remainder}
(\partial_t^2- \Delta)R_N(x,t)= \left( \Delta \, a_N\right)(t-z)^N_+, \qquad R_N \, |_{\left\{t<-1\right\}}=0.
\end{equation}
Note that $a_k\in C^{m- 2k}(\mathbb{R}^n)$ and
\[
\norm{a_k}_{L^\infty(\R^n)} \lesssim \norm{f}_{C^{m-2k}(\R^n)}.
\]
Setting 
\[
\beta= \min\bk{m-2N-2, N-1},
\]
we deduce that $r (y,z,t):= \left( \Delta \, a_N\right)(t-z)^N_+$ belongs to $C^\beta(\mathbb{R}^{n+1})$.
In particular, $r$ belongs to $H^{\beta_1}_{loc}(\mathbb{R}; H^{\beta_2}_{loc}(\mathbb{R}^n))$ with $\beta_j\geq 0$ and $\beta_1+\beta_2=\beta$. Here both $\beta_1$ and $\beta_2$ will be fixed later.
Since $\partial_t^2-\Delta$ is a strictly hyperbolic operator, \cite[Theorems 9.3.1 and 9.3.2]{hormander1976linear}
%\sq{[say more details on how these two theorems are applied to our problem]}
ensures that there exists a unique solution $R_N \in H^{\beta_1+1}_{loc}(\mathbb{R}; H^{\beta_2}_{loc}(\mathbb{R}^n))$ to \eqref{eq:remainder} such that for any given $T>-1$ we have
\begin{equation}\label{r_n:z_N}
\left\|R_N\right\|_{H^{\beta_1+1}((-1, T]; H^{\beta_2}(\mathbb{R}^n))}\lesssim \left\|\left( \Delta \, a_N\right)(t-z)^N_+\right\|_{H^{\beta_1}((-1, T]; H^{\beta_2}(\mathbb{R}^n))}.
\end{equation}
We claim that $R_N\in C^2(\R^{n+1})$ by suitably choosing the parameters  $\M$, $N$, $\beta_1$, and $\beta_2$. Indeed, by \cite[Theorem B.2.8/Vol III]{hormander1983linear}, this follows if for instance
\[
\frac{n+3}{2}< \beta_1 + \beta_2=\beta \quad \text{and} \quad \frac{3}{2}<\beta_1,
\]
and furthermore
\begin{equation}\label{sob_emb} 
\begin{aligned}
\norm{R_N}_{C^2((-1,T]\times \R^n)}
& \lesssim \left\|R_N\right\|_{H^{\beta_1+1}((-1, T]; H^{\beta_2}(\mathbb{R}^n))}\\
& \lesssim \left\|\left( \Delta \, a_N\right)(t-z)^N_+\right\|_{H^{\beta_1}((-1, T]; H^{\beta_2}(\mathbb{R}^n))}.
\end{aligned}
\end{equation}

Equating the parameters involved in the definition of $\beta$, that is, $m-2N-2=N-1$; allow us to choose $m=3N+1$, and hence $\beta=N-1$. We distinguish two cases:
\begin{itemize}
\item When $n$ is even  we consider 
\[
N =\frac{n+6}{2}, \quad m=\frac{3}{2}n +10, \quad  \beta_1=2, \quad \beta_2=\frac{n}{2}. 
\]
\item When $n$ is odd we consider 
\[
N =\frac{n+7}{2}, \quad m=\frac{3}{2}(n+1) +10, \quad  \beta_1=2, \quad \beta_2=\frac{n+1}{2}. 
\]
\end{itemize}
The desired claim is proved by combining the above choices with \eqref{sob_emb}. On the other hand, by \eqref{pr:wave_exp} we derive that $U$ can be written as
\(
U(x,t)= u(x,t)H(t-z),
\)
where clearly 
\[
u(x,t)=  \sum_{j=0}^N a_j (x)(t- z)_+^j + R_N(x,t) 
\]
is of class $C^2$  in the region $\bk{t\geq z}$.  This   shows that $u$ satisfies all the properties stated in Proposition \ref{prop:sol_smooth_1}. We point out that taking $s_0$ large enough we can make another choice of the parameters $\beta_1$, $\beta_2$, $m$ and $N$ so that $u$ will be in $C^K$ in the region $\bk{t\geq z}$ for arbitrary $K\geq 3$. This finishes the proof of Proposition \ref{prop:sol_smooth_1}.

\smallskip

For further purposes, and since $f$ is compactly supported in $\R^n$, we take advantage of the representation of the solution to \eqref{pr:wave_exp} to deduce for all $\alpha, \beta\in \mathbb{N}$ that
\[ 
\norm{a_j (t-z)^j_+}_{H^\alpha((-T, T); H^\beta(\R^n))}|_{\{ t > z \}} \lesssim \norm{f}_{H^{2j+ \beta}(\R^n)}, \quad j=0,1, \ldots, N.
\]
Above we used the representation of the coefficients $a_j$ given by \eqref{a_0:z}-\eqref{a_0:z_z}. Analogously to estimate \eqref{r_n:z_N} we also have  for all $\alpha, \beta\in \mathbb{N}$ that
\[
\norm{R_N}_{H^\alpha((-T, T); H^\beta(\R^n))}|_{\{ t > z \}} \lesssim \norm{f}_{H^{2N+ \beta +2}(\R^n)}.
\]
Combining these estimates with the trace theorem, we conclude for all $\alpha, \beta\in \mathbb{N}$ that
\begin{equation}\label{cont:A_o_f}
\norm{U}_{H^{\alpha}((-T,T); H^{\beta-1/2}({\mathbb{S}}^{n-1}))|_{\{ t > z\}}} \lesssim  \norm{f}_{H^{2N+ \beta +2}(\R^n)}.
\end{equation}

\subsubsection*{\bf Relation between $\mathcal{A}$ and $\tilde{\mathcal{A}}$}

Recall that $\tilde{\mathcal{A}}(\eta-1) = \tilde{\mathcal{U}}|_{\Sigma_T}$ and ${\mathcal{A}}(\eta-1) = {\mathcal{U}}|_{\Sigma_T}$, where $\tilde{\mathcal{U}} = \p_t^2\s  \mathcal{U}$ and both $\tilde{\mathcal{U}}$ and $\mathcal{U}$ vanish for $t < -1$. The following result can be used to estimate $\mathcal{A}$ terms of $\tilde{\mathcal{A}}$ and vice versa. This is quite standard, but one needs some care in the case of Sobolev spaces with negative smoothness index.

\begin{Lemma} \label{lemma_u_dttu}
Let $\mathcal{M}$ be a compact smooth manifold without boundary. Let $T > 1$, let $k \in \mathbb{Z}$, and let $\alpha \in \R$. There is $C > 0$ such that 
\[
C^{-1} \norm{u}_{H^{k+2}((-T,T), H^{\alpha}(\mathcal{M}))} \leq \norm{\p_t^2 u}_{H^{k}((-T,T), H^{\alpha}(\mathcal{M}))} \leq C \norm{u}_{H^{k+2}((-T,T), H^{\alpha}(\mathcal{M}))}
\]
for any $u \in H^{k}((-T,T), H^{\alpha}(\mathcal{M}))$ satisfying $u|_{\{t < -1\}} = 0$.
\end{Lemma}

The proof uses the following Poincar\'e type inequality.

\begin{Lemma} \label{lemma_poincare_type_inequality}
Let $a < a' < b$ and let $k \in \mathbb{Z}$. There is $C > 0$ so that 
\[
\norm{u}_{H^{k+1}((a,b))} \leq C \norm{\p_t u}_{H^k((a,b))}, \qquad u \in C^{\infty}_c((a',b]).
\]
\end{Lemma}
\begin{proof}
We begin with a simple Poincar\'e inequality: for any $u \in C^{\infty}_c((a,b])$ one has 
\[
\int_a^b u^2 \,dt = \int_a^b \p_t(t-b) u^2 \,dt = -2 \int_a^b (t-b) u \p_t u \,dt \leq \frac{1}{2} \int_a^b u^2 \,dt + 2 \int_a^b (t-b)^2 (\p_t u)^2 \,dt.
\]
Absorbing one term to the left hand side gives the inequality 
\[
\norm{u}_{L^2((a,b))} \leq C \norm{\p_t u}_{L^2((a,b))}, \qquad u \in C^{\infty}_c((a,b]).
\]
This shows that for any $k \geq 0$ 
\[
\norm{u}_{H^{k+1}((a,b))}^2 = \sum_{j=0}^{k+1} \norm{\p_t^j u}_{L^2((a,b))}^2 \leq C \norm{\p_t u}_{H^k((a,b))}^2.
\]

Suppose now that $k \leq -1$. If $v \in C^{\infty}_c((a,b))$, define  
\[
w(t) = -\int_t^b v(s) \,ds.
\]
Then $\p_t w = v$. Using that $\abs{w(t)} \leq C \norm{v}_{L^2((a,b))}$ for $t \in [a,b]$, one has 
\[
\norm{w}_{H^{-k}((a,b))} \leq C \norm{v}_{H^{-k-1}((a,b))}.
\]
For any $u \in C^{\infty}_c((a,b])$ and $v \in C^{\infty}_c((a,b))$, since $w(b) = 0$ we have 
\[
\int_a^b u v \,dt = \int_a^b u \p_t w \,dt = -\int_a^b (\p_t u) w \,dt.
\]
Here $w$ is in $H^{-k}_0([a,b))$ but not necessarily in $H^{-k}_0((a,b))$, 
%\tbl{[what is the difference between both spaces?]}
so we cannot directly use duality to get a bound in terms of $\norm{\p_t u}_{H^k}$. However, if we use the fact that $u$ vanishes for $a \leq t \leq a'$, we have 
\[
\int_a^b (\p_t u) w \,dt = \int_a^b (t-a)^{-\abs{k}} (\p_t u) (t-a)^{\abs{k}} w \,dt \leq \norm{(t-a)^{-\abs{k}} \p_t u}_{H^k} \norm{(t-a)^{\abs{k}} w}_{H^{-k}_0}.
\]
Now $\norm{(t-a)^{-\abs{k}} \p_t u}_{H^k} \leq C \norm{\p_t u}_{H^k}$ and $\norm{(t-a)^{\abs{k}} w}_{H^{-k}_0} \leq C \norm{w}_{H^{-k}} \leq C \norm{v}_{H^{-k-1}}$. We have proved that 
\[
\int_a^b u v \,dt \leq C \norm{\p_t u}_{H^k} \norm{v}_{H^{-k-1}}.
\]
Since $H^{k+1}((a,b))$ is the dual of $H^{-k-1}_0((a,b))$, we obtain the desired inequality also for $k \leq -1$.
\end{proof}

\begin{proof}[Proof of Lemma \ref{lemma_u_dttu}]
By density argument, it is enough to prove the statement is true for $u \in C^{\infty}([-T,T], C^{\infty}(\mathcal{M}))$ with $u|_{\{ t < -1 \}} = 0$. Clearly for $k \geq 0$ 
\[
\norm{\p_t^2 u}_{H^k((-T,T), H^{\alpha})}^2 \leq \norm{u}_{H^{k+2}((-T,T), H^{\alpha})}^2
\]
and by duality for $k \leq -2$ 
\[
\norm{\p_t^2 u}_{H^{k}((-T,T), H^{\alpha})} = \sup_{\norm{v}_{H^{-k}_0} = 1} (\p_t^2 u, v) = \sup_{\norm{v}_{H^{-k}_0} = 1} (u, \p_t^2 v) \leq C \norm{u}_{H^{k+2}((-T,T), H^{\alpha})}.
\]
For the remaining case $k=-1$, note that 
\[
\norm{\p_t^2 u}_{H^{-1}((-T,T), H^{\alpha})} = \sup_{\norm{v}_{H^{1}_0} = 1} (\p_t^2 u, v) = \sup_{\norm{v}_{H^{1}_0} = 1} (\p_t u, \p_t v) \leq C \norm{\p_t u}_{L^2((-T,T), H^{\alpha})}
\]
and use that $\norm{\p_t u}_{L^2((-T,T), H^{\alpha})} \leq \norm{u}_{H^1((-T,T), H^{\alpha})}$.

To prove the converse inequality, let $g$ be some Riemannian metric on $\mathcal{M}$ and let $(\varphi_l)_{l=1}^{\infty}$ be an orthonormal basis of $L^2(\mathcal{M})$ consisting of eigenfunctions of $-\Delta_g$. Write $u_l(t) = (u(t), \varphi_l)_{L^2(\mathcal{M})}$ and $\br{l} = (1+l^2)^{1/2}$. If $m \geq 0$ one has 
\begin{align}
\norm{u}_{H^m((-T,T), H^{\alpha})}^2 &= \sum_{j=0}^m \int_{-T}^T \norm{\p_t^j u(t)}_{H^{\alpha}}^2 \,dt = \sum_{j=0}^m \sum_{l=1}^{\infty} \int_{-T}^T \br{l}^{2\alpha} \abs{\p_t^j u_l(t)}^2 \,dt \notag \\
 &= \sum_{l=1}^{\infty} \br{l}^{2\alpha} \norm{u_l}_{H^m((-T,T))}^2. \label{u_hm_halpha_norm}
\end{align}
For $m \leq -1$, using \eqref{u_hm_halpha_norm} to $\norm{v}_{H^{-m}((-T,T), H^{-\alpha})}$ gives 
\begin{align*}
\norm{u}_{H^m((-T,T), H^{\alpha})} &= \sup_{\norm{v}_{H^{-m}_0((-T,T), H^{-\alpha})} = 1} (u, v)_{L^2} \\
 &=  \sup_{v \in H^{-m}_0, \ \sum_l \br{l}^{-2\alpha} \norm{v_l}_{H^{-m}}^2 = 1} \int_{-T}^T \sum_{l=1}^{\infty} \br{l}^{\alpha} u_l(t) \br{l}^{-\alpha} v_l(t) \,dt \\
 &\leq \sup_{v \in H^{-m}_0, \ \sum_l \br{l}^{-2\alpha} \norm{v_l}_{H^{-m}}^2 = 1} \sum_{l=1}^{\infty} \norm{\br{l}^{\alpha} u_l}_{H^m((-T,T))} \norm{\br{l}^{-\alpha} v_l}_{H^{-m}_0((-T,T))} \\
 &\leq \Big( \sum_{l=1}^{\infty} \br{l}^{2\alpha} \norm{u_l}_{H^m((-T,T))}^2 \Big)^{1/2}.
\end{align*}
Choosing $v_l$ so that 
\[
\norm{\br{l}^{-\alpha} v_l}_{H^{-m}_0((-T,T))} = \frac{\norm{\br{l}^{\alpha} u_l}_{H^m((-T,T))} }{\left( \sum_{l=1}^{\infty} \br{l}^{2\alpha} \norm{u_l}_{H^m((-T,T))}^2 \right)^{1/2}}
\]
gives that 
\begin{equation} \label{u_mixed_norm_fourier}
\norm{u}_{H^m((-T,T), H^{\alpha})}^2 =  \sum_{l=1}^{\infty} \br{l}^{2\alpha} \norm{u_l}_{H^m((-T,T))}^2
\end{equation}
for any $m \in \mathbb{Z}$.

Now let $k \in \mathbb{Z}$. Since $u_l(t) = 0$ for $t < -1$, we obtain from \eqref{u_mixed_norm_fourier} and Lemma \ref{lemma_poincare_type_inequality} (used twice) that 
\[
\norm{u}_{H^{k+2}((-T,T), H^{\alpha})}^2 \leq C \sum_{l=1}^{\infty} \br{l}^{2\alpha} \norm{\p_t^2 u_l}_{H^k((-T,T))}^2 \leq C \norm{\p_t^2 u}_{H^k((-T,T), H^{\alpha})}^2.
\]
This concludes the proof.
\end{proof}

%\bibliography{reference_bib_LS2020}
%\bibliographystyle{alpha}

{


}

\end{document}